\documentclass[reqno]{amsart}
\usepackage{enumitem}
\usepackage{color}
\usepackage{amsmath}
\usepackage{amssymb}
\numberwithin{equation}{section}
\allowdisplaybreaks
\usepackage{enumitem}
\newtheorem{definition}{Definition}
\newtheorem{theorem}{Theorem}

\newtheorem{rem}{Remark}

\newtheorem{lemma}{Lemma}
\usepackage{framed}
\newcommand{\loc}{\operatorname{loc}}

\newcommand{\eps}{\varepsilon}
\usepackage{graphicx}
\usepackage{calc}
\usepackage{todonotes}

\newcommand{\llangle}[1][]{\savebox{\@brx}{\(\m@th{#1\langle}\)}%
  \mathopen{\copy\@brx\kern-0.5\wd\@brx\usebox{\@brx}}}
\newcommand{\rrangle}[1][]{\savebox{\@brx}{\(\m@th{#1\rangle}\)}%
  \mathclose{\copy\@brx\kern-0.5\wd\@brx\usebox{\@brx}}}

\usepackage{accents}

\title[Nonlinear Milne Problem for radiative heat transfer system]{Stability of the nonlinear Milne Problem for radiative heat transfer system}
\author{Mohamed Ghattassi}
\address{Department of Mathematics, New York University in Abu Dhabi, Saadiyat Island, P.O. Box 129188, Abu Dhabi, United Arab Emirates, {\sf mg6888@nyu.edu}}
\author{Xiaokai Huo}
\address{Institute for Analysis and Scientific Computing, Vienna University of Technology, Wiedner Hauptstraße 8–10, 1040 Wien, Austria, {\sf xiaokai.huo@tuwien.ac.at}}
\author{Nader Masmoudi}
\address{Department of Mathematics, New York University in Abu Dhabi, Saadiyat Island, P.O. Box 129188, Abu Dhabi, United Arab Emirates-
Courant Institute of Mathematical Sciences, New York University, 251 Mercer Street, New York, NY 10012, USA, {\sf nm30@nyu.edu}}

\begin{document}

\begin{abstract} 
    This paper focuses on the nonlinear Milne problem of the radiative heat transfer system on the half-space. The nonlinear model is described by a second order ODE for temperature coupled to transport equation for radiative intensity. The nonlinearity of the fourth power Stefan-Boltzmann law of black body radiation, bring additional difficulty in mathematical analysis, compared to the well-developed theory for Milne problem of linear transport equation. With the help of the monotonicity property of the second order ODE, we prove the existence of the nonlinear Milne problem on a finite interval using monotonic convergence theorems. Then the solution is extended to the half-space using a uniform weighted estimate and the compactness method. Moreover, the solutions are proved to converge to constants as $x\to \infty$. Therefore, the linear stability analysis is used to study the uniqueness of the nonlinear Milne problem. The existence and uniqueness for the linearized system is established under a spectral assumption on the solution of the nonlinear problem. The spectral assumption is shown to be satisfied when the boundary data is close to the well-prepared case by using a generalized Hardy's inequality. The uniqueness of the solution to the half-line nonlinear Milne problem is established in a neighborhood of solutions satisfying a spectral assumption and the energy estimate. 
    The current work extends the study of Milne problem for linear transport equations and provides a comprehensive study on the nonlinear Milne problem of radiative heat transfer systems.

\end{abstract}

\maketitle

\tableofcontents

\section{Introduction}

We consider the following system
\begin{align}
    \partial_{x}^2 T + \langle \psi - T^4 \rangle  =~& 0,\label{eq:m1}\\
    \mu \partial_x \psi + (\psi-T^4) = ~&0,\label{eq:m2}
\end{align}
in the space $(x,\mu) \in \mathbb{R}_+ \times [-1,1]$,
where $T=T(x)$ and $\psi=\psi(x,\mu)$ are the temperature and radiative intensity, respectively.
 The bracket $\langle \cdot \rangle$ denotes the integral over $\mu \in [-1,1]$, i.e. $\langle f\rangle := \int_{-1}^1 f(\mu)d\mu$.
The boundary conditions are taken to be the non-homogeneous Dirichlet conditions
\begin{align}
    T(0) =~& T_b, \label{eq:bd1}\\
    \psi(0,\mu) =~&\psi_b(\mu),  \text{ for any }\mu \in(0,1], \label{eq:bd2}
\end{align}
where $T_b \ge 0$ is a given constant and $\psi_b=\psi_b(\mu) \ge 0$ is a given function of $\mu \in (0,1].$
Here $\psi_b$ describes the radiative intensity transmitted into the medium from outside.
In this paper, the above problem is called the \emph{nonlinear Milne problem of radiative heat transfer}. The main objective of our work is to study the existence, uniqueness, long time behavior and stability of this nonlinear Milne problem.

System \eqref{eq:m1}-\eqref{eq:m2} was derived in the boundary layer problem for the diffusive limit of the following radiative heat transfer system
\begin{align*}
    &\partial_t T^\eps = \Delta T^\eps + \frac{1}{\eps^2}\langle \psi^\eps - (T^\eps)^4 \rangle,\\
    &\partial_t \psi^\eps + \frac{1}{\eps}\omega\cdot \nabla \psi^\eps = -\frac{1}{\eps^2}(\psi^\eps-(T^\eps)^4),
\end{align*}
where $T^\eps=T^\eps(t,X)$, $\psi^\eps=\psi(t,X,\omega)$ with $X\in\mathbb{R}^3$, $\omega \in \mathbb{S}^2$. System \eqref{eq:m1}-\eqref{eq:m2} is the boundary layer problem of the corresponding steady state problem of the above radiative heat transfer system, for more details we refer the reader to \cite{klar2001numerical} and references therein.
When the boundary data is well-prepared such that $\psi_b=T_b^4$, system \eqref{eq:m1}-\eqref{eq:m2} has the constant solution $(T_b,\psi_b1_{\mu \in [0,1]})$, where $1_{\mu \in [0,1]}$ denotes the characteristic function on $[0,1]$, and
no boundary layer exists. In this case the diffusive limit for the radiative transfer system was rigorously justified in our previous work \cite{ghattassi2020diffusive} based on the weak convergence method and the relative entropy method. In general, when $\psi_b \neq T_b^4$, it is crucial to study system \eqref{eq:m1}-\eqref{eq:m2} because the boundary layer effects make the interior solution obtained by asymptotic analysis invalid near the boundary. The solution to \eqref{eq:m1}-\eqref{eq:m2} provides a better description of the radiative transfer system near the boundary. Moreover, the boundary condition for the interior solutions needs to be determined by the limit of the solutions to the nonlinear Milne problem \eqref{eq:m1}-\eqref{eq:m2} as $x\to \infty$. Based on this work, we are able to provide a rigorous justification of the diffusive limit \cite{Bounadrylayer2021GHM3,Bounadrylayer2021GHM4}. Hence it is important to study the nonlinear Milne problem and show the behavior of its solutions at infinity.

There has been a well-developed theory of the Milne problem for the linear transport equations. When the equation for temperature is not considered, system \eqref{eq:m1}-\eqref{eq:m2} becomes 
\begin{align}\label{eq:lm}
    \mu \partial_x \psi + \psi - \frac12 \langle \psi \rangle = 0.
\end{align}
This problem, referred to as the Milne problem for linear transport equation, dates back to the work \cite{bensoussan1979boundary,bardos1984diffusion}. It was derived in \cite{bensoussan1979boundary} as the boundary layer problem of the linear transport equation. The equation was shown to decay exponentially to a constant therein by using the theory of stochastic processes. Later, this was also proved by using compactness method and weighted estimates in \cite{bardos1984diffusion}. Since $0$ belongs to the spectral of the linear operator $\mathcal{L}\psi:=-\mu\partial_x\psi - \psi + \tfrac12\langle\psi\rangle$, one cannot use semigroup theory to show the existence of solutions for the above problem. Instead, the maximum principle is used to show the existence of solutions in $L^\infty$. Moreover, using a weighted estimate, the exponential convergence and the uniqueness of solutions is shown in \cite{bardos1984diffusion}. Recently, a geometric correction to the above Milne problem is proposed in \cite{wu2015geometric} to better describe the behavior of solutions near the boundary when the boundary is not flat. A weighted estimate accounting for the geometric corrections was provided therein to show the exponential convergence of solutions.

Compared to the Milne problem for linear transport equations, there are only few works on the nonlinear Milne problem. For example, in \cite{klar2001numerical}, the existence of weak solutions in $L^\infty_{\rm loc}(\mathbb{R}_+)\times L^\infty_{\rm loc}(\mathbb{R}_+\times [-1,1])$ was shown under the assumption that $\gamma_1\le T_b \le \gamma_2$ for some constants $\gamma_1>0$, $\gamma_2 \ge 0$ by using the Leray-Schauder fixed point theorem. However, whether the solutions converges to some finite constants is not known. Moreover, the uniqueness of weak solutions are not known. In \cite{kelley1996existence}, existence and uniqueness are shown for system \eqref{eq:m1}-\eqref{eq:m2} with Dirichlet boundary data in a bounded interval by using the maximum principle. For the half-space nonlinear Milne problem, the boundary condition at infinity is not known and such method cannot be adapted to our problem.

The major difficulty of the study of the nonlinear Milne problem \eqref{eq:m1}-\eqref{eq:bd2} lies in two aspects. Firstly, the second order ODE \eqref{eq:m1} and the transport equation \eqref{eq:m2} are coupled and need to be studied together. Maximum principle holds for equation \eqref{eq:m1} and \eqref{eq:m2} separately, but is not known to hold for the system. Secondly, the $T^4$ term is nonlinear, which is due to the Stefan-Boltzmann law of black body radiation. This physical term, makes the Milne problem nonlinear and the theory of linear operators  not applicable. The lack of maximum principle and the nonlinearity nature of the problem, make the study of our problem challenging. 


To overcome the above difficulty, several techniques are used. Firstly, the maximum principle of the second order ODE \eqref{eq:m1} and the transport equation \eqref{eq:m2} also implies the monotonicity of solutions with respect to source term and boundary data. Such monotonicity property could help to construct an approximate sequence converging to the solutions to \eqref{eq:m1}-\eqref{eq:bd2}. Such a technique was used in \cite{clouet2009milne} to study the Milne problem for a similar system but without the second order derivative $\partial_x^2 T$. Here we need to use the monotonicity property of the second order ODE to construct functions that converge to solutions of the nonlinear Milne problem. Secondly, to prove the exponential decay property of the solution, we cannot directly follow \cite{bardos1984diffusion} since the property that $\partial_x \psi$ satisfies the same equation as $\psi$ is used to show the exponential decay of solution for the Milne problem \eqref{eq:lm}. However, in our case, due to nonlinearity, $(\partial_x T,\partial_x \psi)$ satisfies a different system.
Fortunately, we are able to derive a weighted estimate in a similar way as \cite{bardos1984diffusion} that leads to the exponential decay of $T$ as well as $\partial_x T$, which also implies the decay property of $\psi$ by directly using the formulas for linear transport equation. Thirdly, unlike the linear Milne problem \eqref{eq:lm}, here the uniqueness of nonlinear Milne problem \eqref{eq:m1}-\eqref{eq:bd2} does not follow from the weighted estimates and decay property of solutions due to the nonlinearity. We solve this issue by using linear stability analysis and propose a spectral assumption that leads to weighted estimates and the uniqueness of the linearized system. Lastly, a sufficient condition is proposed for the spectral assumption by using the generalized Hardy's inequality. This allows us to prove the uniqueness of the nonlinear Milne problem in the vicinity of solutions satisfying the sufficient condition or near well-prepared boundary data where  $\psi_{b}-T_b^4$ is small.


System \eqref{eq:m1}-\eqref{eq:m2} could be extended by considering the dependence on frequency of the radiative density. The corresponding frequency dependent model reads as
\begin{align*}
    &\partial_x^2 T +   \int_0^\infty \int_{-1}^1 (\psi - B_\nu(T)) d\mu d\nu = 0,\\
    &\mu \partial_x \psi + \psi - B_\nu(T) = 0,
\end{align*}
where $T=T(x),\psi = \psi(x,\mu,\nu)$ and $\nu \in (0,\infty)$ is the frequency. The nonlinear function $B_\nu(T)$ is usually taken to be the Planck function 
\[
B_\nu(T) = \frac{2h\nu^3}{c^2}\frac{1}{e^{\frac{h\nu}{kT}}-1},
\]
where $c$ is the speed of light, $k$ is the Boltzmann constant and $h$ is the Planck constant. When the diffusive operator $\partial_x^2 T$ is neglected, the Milne problem is studied in \cite{larsen1983asymptotic,sentis1987half,clouet2009milne}. The existence of the above frequency dependent half-space problem was proved in \cite{sentis1987half}, where more later a simplified demonstration was provided in \cite{clouet2009milne}. However, the existence and uniqueness for the nonlinear Milne problem \eqref{eq:m1}-\eqref{eq:m2} and the above frequency dependent problem has not been proved in the literature. Our result of this paper could be extended to the frequency dependent case without any difficulty.


Our work could be treated as a development of the method of weighted estimates first developed in \cite{bardos1984diffusion} to the nonlinear coupled Milne problem. The work \cite{bardos1984diffusion} paves the way to use rigorous mathematical analysis to study the boundary layer problem of kinetic PDEs. The same technique was used in \cite{wu2015geometric} to study the Milne problem of linear transport equation with geometric corrections. Such techniques have been also used and developed for the other kinetic equations, in particular Boltzmann equations. For example, existence, uniqueness and asymptotic behavior of the linearized Boltzmann equation for a gas of hard spheres is studied by using more complicated weighted estimates in \cite{bardos1986milne}. The analysis is extended to gas mixtures in \cite{aoki2003knudsen} and to the Vlasov-Boltzmann-Poisson equation in \cite{bostan2010boundary}. The nonlinear Milne problem for Boltzmann equation was considered in \cite{golse1988boundary}, where existence and exponential decay property of solutions are proved for the  Milne problem for the nonlinear Boltzmann equation in one-dimension with a slightly perturbed reflection boundary condition using weighted estimates and the Banach fixed-point theorem.  For a boundary condition that the gas is in contact with its condensed phase, the existence and uniqueness of a uniformly decaying solution is shown in a neighborhood of Maxwellian equilibrium in \cite{bernhoff2021boundary}, where a penalized problem is studied in order to consider the phase transition at the interface. Besides these works, there are also some other ways to study the Milne problems of Boltzmann equations, for example \cite{aoki2003knudsen,bostan2010boundary,huang2022boundary,liu2013invariant,ukai2003nonlinear}.  Among them,  In \cite{liu2013invariant}, the authors use the theory of invariant manifolds and use Green functions to construct linear invariant manifolds for the half-space Milne problem of Boltzmann equation with Dirichlet boundary conditions. Based on the macro-micro or hydrodynamics-kinetic decomposition of solutions, existence of solutions to the same nonlinear problem is proved in \cite{ukai2003nonlinear} and it is found that the Mach number of Maxwellian affects the existence. For the case of reflective boundary conditions, the existence is shown recently in \cite{huang2022boundary}, where uniqueness is also established under some constraint conditions. For a review of the boundary layer problem for Boltzmann equation, we refer the reader to \cite{bardos2006half}. Since techniques and methods for the Milne problem of linear transport equations have been extended to study the Milne problem for the Boltzmann equation, for example from \cite{bardos1984diffusion} to \cite{bardos1986milne} and \cite{bernhoff2021boundary}, the study of the nonlinear Milne problem \eqref{eq:m1}-\eqref{eq:bd2} could provide tools and insights for other linear and nonlinear Milne problem for system of kinetic equations.


\subsection{Main results}
The existence and exponential decay properties of solutions to the nonlinear Milne problem \eqref{eq:m1}-\eqref{eq:bd2} is stated in the following theorem.
\begin{theorem}[Existence of the nonlinear Milne problem]\label{thm:ex}
    Given $T_b\ge 0$, $\psi_b=\psi_b(\mu)\ge 0$ for any $\mu \in(0,1]$, there exists a bounded non-negative weak solution $\left(T,\psi \right) \in L_{\rm loc}^2(\mathbb{R}_{+})\times L_{\rm loc}^2(\mathbb{R}_{+}\times [-1,1])$ to system \eqref{eq:m1}-\eqref{eq:m2}, and the solution satisfies the following estimate:
    \begin{align}\label{eq:es-2}
        &\int_0^{\infty} e^{2\alpha x} 4T^3|\partial_x T|^2 dx + (1-\alpha) \int_0^{\infty}\int_{-1}^1 e^{2\alpha x} (\psi-T^4)^2 d\mu dx \nonumber\\
        &\quad + \frac{1}{2}\int_{-1}^0 |\mu| (\psi(0,\mu)-T_b^4)^2 d\mu \le \frac{1}{2} \int_0^1 \mu (\psi_b-T_b^4)^2 d\mu,
    \end{align}
 for any $0 \le \alpha <1$. Moreover, there exists a constant $T_\infty \ge 0$ such that 
 \begin{equation}
   \begin{aligned} \label{eq:thm.ex}
    	&|T(x)-T_\infty|\le M_{\alpha} e^{-\alpha x},\\
        & |\psi(x,\mu)-T_\infty^4| \le |\psi_b - T_b^4| e^{-\frac{x}{\mu}} + 4(T_b+2M_{\alpha})^3M_{\alpha} \frac{1}{1- \mu\alpha}e^{-\alpha x},
        ~ \text{for } \mu>0, \\
        &|\psi(x,\mu)-T_\infty^4| \le 4(T_b+2M_{\alpha})^3M_{\alpha}\frac{1}{1-\mu\alpha} e^{-\alpha x},\quad \text{for } \mu<0,
    \end{aligned} 
\end{equation}
    for any $x\in \mathbb{R}_+$ and $0\le \alpha <1$, where $$M_{\alpha}=\frac{1}{\sqrt{6\alpha(1-\alpha)}}\left(\int_0^1 \mu(\psi_b-T_b^4)^2 d\mu\right)^{\frac12}.$$
    As a consequence, we can get the estimate 
    \begin{align}\label{eq:Tbd}
        T_\infty \le T_b + M_{\alpha},\quad |T(x)| \le T_b + 2M_{\alpha}.
    \end{align}
    
    

    \end{theorem}
The above theorem gives the existence of bounded solutions for system \eqref{eq:m1}-\eqref{eq:m2} and shows that $T$ converges exponentially to some nonnegative constant $T_\infty$ as $x \to \infty$, and provides an estimate of the limit of solutions at infinity. However, the uniqueness of solutions is not guaranteed by this theorem.

 The uniqueness of solutions to system \eqref{eq:m1}-\eqref{eq:m2} is related to the linear stability of the system. The linear stability relies on the following spectral assumption.
 \begin{framed} \textbf{\underline{Spectral Assumption} :}
    \begin{enumerate}[label=({\bf\Alph*})]
        \item \label{asA} 
        We say that the function $T \in C^1(\mathbb{R}_{+})$ satisfies the spectral assumption if there exists  constants $\beta_{0}>0$ and $0<M<1$, such that 
        \begin{align}
            M\int_0^\infty e^{2\beta_{0}x} (2T^{\frac32})^2 |\partial_x f|^2 dx  \ge  4\int_0^\infty  e^{2\beta_{0} x}  |\partial_x (2T^{\frac32})|^2 f^2 dx 
            \label{eq:spassump2}
        \end{align}
     for all measurable function $f \in C^{1}(\mathbb{R}_{+})$ with $f(0)=0$.
    \end{enumerate}
    \end{framed}
    To show how this spectral assumption implies linear stability, we study the following system, which is the linearization of system \eqref{eq:m1}-\eqref{eq:m2}.
     \begin{align}
    &\partial_x^2 g + \langle \phi - 4T^3 g \rangle=\langle S_1 \rangle,\label{eq:l-1/i}\\
    &\mu\partial_x \phi + (\phi - 4T^3 g ) = S_1,\label{eq:l-2/i}
\end{align}
supplemented with the following boundary conditions
\begin{align}
    &g(0) = 0, \label{eq:l-1b/i}\\
    &\phi(0,\mu) = \phi_b(\mu), \text{ for any } \mu \in(0,1]\label{eq:l-2b/i},
\end{align}
where $T$ is the solution to  system \eqref{eq:m1}-\eqref{eq:m2} and  the source term $S_1=S_1(x)$ is a given function that vanishes as $x\to \infty$. Define the linear operator $\mathcal{F}:\,\,\,\, L^p(\mathbb{R}_+)\times L^p(\mathbb{R}_+\times[-1,1]) \longrightarrow L^p(\mathbb{R}_+)\times L^p(\mathbb{R}_+\times[-1,1])$ (with $1\le p \le \infty$), given by  
\begin{align}\label{eq:opF}
    \mathcal{F}(g,\phi) := \left(\begin{array}{cc}
        \partial_x^2 g + \langle \phi - 4T^3 g\rangle \\
       - \mu\partial_x\phi - (\phi-4T^3 g)
    \end{array}\right).
\end{align} 
Then it will be shown later that the spectral assumption implies that the spectrum of $\mathcal{F}$ all lie in the space $\Re z \le 0$. However, similar as the linear transport operator studied in \cite{bardos1984diffusion}, we can derive a weighted estimate on the solutions of \eqref{eq:l-1/i}-\eqref{eq:l-2/i} that leads to the exponentially decay. The weighted estimate relies on the following inequality 
\begin{align}
    \int_0^\infty e^{2\beta x} 4T^3 |\partial_x f|^2 dx + \int_0^\infty \partial_xe^{2\beta x} (4T^3)f\partial_x f dx \ge 0, 
\end{align}
for any $f\in C^1_{\rm loc}(\mathbb{R}_+)$ satisfying $f(0)=0$, where $\beta$ is the same as in the spectral assumption \ref{asA}. It will be proved that the above inequality is a direct consequence of the spectral assumption \ref{asA} (see Lemma \ref{LemmaSA} for the details). The weighted estimates allows us to prove the 
 existence, uniqueness and exponential decay for solutions to the linearized system \eqref{eq:l-1/i}-\eqref{eq:l-2/i}.
\begin{theorem}\label{thm2}
    Suppose $T$ satisfies the spectral assumption \ref{asA} for some constants $0<\beta<1$ and $M>0$. Assume $S_1$ decays to zero exponentially such that $\int_0^\infty e^{2\beta x} |S_1|^2 dx <\infty$.
    Then, there exists a unique bounded solution $\left(g,\phi\right) \in L^2_{\rm loc}(\mathbb{R}_+)\times L^2_{\rm loc}(\mathbb{R}_{+}\times [-1,1])$
     to system \eqref{eq:l-1/i}-\eqref{eq:l-2/i} with boundary conditions \eqref{eq:l-1b/i}-\eqref{eq:l-2b/i},
     and the solution satisfies
     \begin{align}\label{eq:el/2}
        &\int_0^\infty \int_{-1}^1  e^{2\beta x} (\phi-4T^3 g)^2 d\mu dx + \frac{1}{1-\beta}\int_{-1}^0 |\mu| (\phi(0,\cdot))^2 d\mu \nonumber\\
        &\qquad\le \frac{1}{1-\beta} \int_0^1 \mu \phi_b^2 d\mu + \frac{2}{(1-\beta)^2} \|e^{\beta x}S_1\|_{L^2(\mathbb{R}_+)}^2, 
     \end{align}
     for any $x\in\mathbb{R}_+$.
    Moreover, there exists a constant $g_\infty \in \mathbb{R}$ 
    such that 
    \begin{equation}\label{eq:thm2.decay}
    \begin{aligned}
    	&|g(x)-g_\infty| \le N_{\beta} e^{-\beta x},\\
        &|\phi(x,\mu) - 4T_\infty^3 g_\infty | \le |\psi_b - 4T_\infty^3 g_\infty| e^{-\frac{x}{\mu}} + 4(T_b+2M_{\alpha})^2(4M_{\alpha}+T_b)N_{\beta} e^{-\beta x}\\
        &\qquad \qquad \qquad \qquad  + \frac{1}{1-2\mu\beta} e^{-\beta x}  \|e^{\beta x}S_1\|_{L^2(\mathbb{R}_+)},
        ~\text{for }\mu>0,\\
        &|\phi(x,\mu) - 4T_\infty^3 g_\infty | \le 4(T_b+2M_{\alpha})^2(4M_{\alpha}+T_b)N_{\beta} e^{-\beta x},\\
        &\qquad \qquad \qquad \qquad  + \frac{1}{1-2\mu\beta} e^{-\beta x}  \|e^{\beta x}S_1\|_{L^2(\mathbb{R}_+)},~\text{for }\mu<0.
    \end{aligned}
\end{equation}
    where $$N_{\beta}=\frac{1}{\sqrt{2\beta}}\left(\frac{2}{3(1-\beta)}\int_0^1\mu\phi_b^2 d\mu + \frac{4}{3(1-\beta)^2} \|e^{\beta x} S_1\|_{L^2(\mathbb{R}_+)}^2 \right)^{\frac12},$$
    where $\alpha$ is the constant in Theorem \ref{thm:ex}. As a consequence,
    \begin{align}
        g_\infty \le N_{\beta}, \quad |g(x)|\le 2N_{\beta}.
    \end{align}
\end{theorem}

Although the spectral assumption \ref{asA} implies the uniqueness for the linearized system \eqref{eq:l-1/i}-\eqref{eq:l-2/i}, we are not able to establish such a uniqueness result for the nonlinear Milne problem \eqref{eq:m1}-\eqref{eq:bd2}. This is partly because the spectral assumption \ref{asA} is not additive, which means if $T_1$ and $T_2$ satisfy the spectral assumption, $T_1+T_2$ is not known to satisfy the spectral assumption. In order to show the uniqueness of the nonlinear Milne problem, we need a more restrictive condition. Inspired by the generalized Hardy's inequality, we propose the following condition:
\begin{flalign}\label{eq:A2}
    \textbf{(A0)}  &&A_0^2 :=  \sup_{r\in (0,\infty)}\left(\int_r^\infty e^{2\beta x} 36 T|\partial_x T|^2 dx \cdot \int_0^r \frac{e^{-2\beta x}}{4T^3} dx\right) < \frac12.
\end{flalign}
By using the generalized Hardy's inequality, this condition is shown to be a sufficient condition for the spectral assumption \ref{asA} (see Lemma \ref{lem.spcond}). Moreover, when $A_0^2 \ge 2$, the spectral assumption fails to hold.

The benefits of the condition \textbf{(A0)} is that if $T_1$ and $T_2$ satisfies \textbf{(A0)}, then $w:=(T_1^4-T_2^4)/(T_1-T_2) = T_1^3 + T_2^3 + T_1^2 T_2 + T_2^2T_1$
satisfies the spectral assumption with $T^{\frac32}$ replaced by  $\tfrac12 w^{\frac12}$. This property, together with the uniqueness of \eqref{eq:l-1/i}-\eqref{eq:l-2/i}, implies the uniqueness for the nonlinear Milne problem \eqref{eq:m1}-\eqref{eq:m2}.

\begin{theorem}\label{thm3}
    Let $(T,\psi)$ and $(T_1,\psi_1)$ be two solutions to the nonlinear Milne problem \eqref{eq:m1}-\eqref{eq:m2} satisfying the weighted estimate \eqref{eq:es-2}. Suppose the function $T$ satisfy the spectral assumption \ref{asA}. Then if $(T_1,\psi_1) \in \mathcal{V}$ with 
    \begin{align}\label{eq:spV}
        \mathcal{V} = \{(G,\phi) &\in L^2_{\rm loc}(\mathbb{R}_+)\times L^2_{\rm loc}(\mathbb{R}_+\times[-1,1]):\nonumber\\
        &\|G-T\|_{L^2(\mathbb{R}_+)} \le \eps\},
    \end{align} 
    for $\eps>0$ sufficiently small. Then, $T_1(x)=T(x)$ and $\psi_1(x,\mu)=\psi(x,\mu)$
     for $(x,\mu)\in \mathbb{R}_+\times[-1,1]$ almost everywhere.
    
    Furthermore, if the condition $(T_1,\psi_1) \in \mathcal{V}$ is replaced by the condition 
    \begin{align}\label{eq:ass2}
        \frac12 \int_0^1 \mu(\psi_b-T_b^4)^2 d\mu \le C_b,
    \end{align}
    for some constant $C_b$ (given in \eqref{eq:Cb}). Then, $T_1(x)=T(x)$ and $\psi_1(x,\mu)=\psi(x,\mu)$
     for $(x,\mu)\in \mathbb{R}_+\times[-1,1]$ almost everywhere, and the solution of the nonlinear Milne problem \eqref{eq:m1}-\eqref{eq:m2} is unique.


\end{theorem}

 Here we are only able to establish the uniqueness of the solution of the nonlinear Milne problem for solutions satisfying the weighted estimate \eqref{eq:es-2} or under the smallness assumption \eqref{eq:ass2}. Indeed, the assumption \eqref{eq:ass2} is satisfied for almost well-prepared boundary data and the uniqueness of the nonlinear Milne problem \eqref{eq:m1}-\eqref{eq:m2} is proved for this case. However, the uniqueness of the solution of the nonlinear Milne problem  for general boundary data remains an open problem.
\subsection{Organizations}
The paper is organized as follows. In section 2, the existence of solutions to system \eqref{eq:m1}-\eqref{eq:m2} is proved in a bounded interval, in Theorem \ref{thm.bd}. To prove this theorem, the monotonicity for the second order ODE is shown in \eqref{eq:m1} in section \ref{subMonot} and Theorem \ref{thm2} is proved in section \ref{SubExUni} for existence and uniqueness, and section \ref{sec:weB}, for the weighted estimate. 
Section \ref{proofTh1} is devoted to  the proof of Theorem \ref{thm:ex}. The existence of solutions and weighted estimates are shown in Section \ref{ProofEx1} and section \ref{ProofWEightEst1}, respectively, both by passing to the limit $B\to\infty$ in the solutions on the bounded interval $[0,B]$. The exponential decay  property \eqref{eq:thm.ex} of solutions is shown in section \ref{ProofExpo1}. Section \ref{LMB0} is devoted to the study of the linearized Milne problem. Followed by a discussion on the spectral assumption \ref{asA} in section  \ref{LMB1}, the existence of weak solutions to the linearized Milne problem on a bounded interval $[0,B]$ is shown in section  \ref{LMB2}. The solution is then extended to the half-space, with existence proved in section \ref{LMB3}, weighted estimate and decay proved in section {LMB4}, uniqueness in \ref{LMB5}, thus proving Theorem \ref{thm2}. 
 Finally, we prove Theorem \ref{thm3} in section \ref{proofThm3}, where the sufficient condition \ref{eq:A2} is shown to imply the spectral assumption \ref{asA} in section \ref{51} and the uniqueness of solutions for system \eqref{eq:m1}-\eqref{eq:m2} is established for the case \eqref{eq:ass2} in section \ref{52} and the case $T_2 \in \mathcal{V}$ in section \ref{smalS} . 

\section{Existence on the bounded interval}\label{section2}
In this section,
we consider system \eqref{eq:m1}-\eqref{eq:m2} on the bounded interval $x\in [0,B]$:
\begin{align}
    &\partial_x^2 T^B + \langle \psi^B - (T^B)^4 \rangle = 0,\label{eq:B1}\\
    &\mu \partial_x \psi^B + \psi^B - (T^B)^4 = 0,\label{eq:B2}
\end{align}
associated to the following boundary conditions 
\begin{align}
    &T^B(0) = T_b, \quad \partial_x T^B(B) = 0,\label{eq:B1b}\\
    &\psi^B(0,\mu) = \psi_b(\mu), \quad \psi^B(B,\mu)  = \psi^B(B,-\mu), \quad \text{ for any }\mu \in (0,1].\label{eq:B2b}
\end{align}
The boundary conditions at $x=B$ are motivated by the fact that solutions to the nonlinear Milne problem \eqref{eq:m1}-\eqref{eq:m2} are expected to converge to some constants at infinity.
If the right boundary conditions at $x=B$ are replaced by Dirichlet boundary conditions, the existence for the above system was proved in \cite{kelley1996existence} by using the fixed point theorem together with the maximum principle. However, here we do not know the value of $T^B,\psi^B$ on the boundary $x=B$. The method of \cite{kelley1996existence} cannot be used to show the existence for the above system. A different approach based on the theory of subsolutions and fixed point theorems is used here to prove the existence of the solution for system \eqref{eq:B1}-\eqref{eq:B2}. The theory of subsolutions is used in \cite{clouet2009milne} to show the existence of solutions to system \eqref{eq:m1}-\eqref{eq:m2} without the diffusion operator $\partial_x^2 T$. 
Here we will study the property of solutions for \eqref{eq:B1} and show the existence for system \eqref{eq:B1}-\eqref{eq:B2}. The existence theorem reads as follows.
\begin{theorem}\label{thm.bd}
    Assume $ 0\le T_b \le \gamma$, $0\le \psi_b=\psi_b(\mu) \le  \gamma^4$ for any $\mu \in(0,1]$  and for some constant $\gamma>0$. Then there exists a unique solution $\left(T^B,\psi^B \right) \in C^2([0,B])\times C^1([0,B]\times [-1,1])$ to system \eqref{eq:B1}-\eqref{eq:B2} with boundary conditions \eqref{eq:B1b}-\eqref{eq:B2b},
and the solution 
satisfies
    \[
    0\le T^B(x) \le \gamma, \quad  0\le \psi^B(x,\mu) \le \gamma^4,\quad \text{for any }x\in[0,B],\, \mu \in [-1,1].
    \]
    Furthermore, the following estimate
    \begin{align}\label{eq:estB2}
        \int_0^{B} & e^{2\alpha x} 4 (T^B)^3 |\partial_x T^B|^2 dx + (1-\alpha )\int_0^{B}\int_{-1}^1 e^{2\alpha x} (\psi^B - (T^B)^4)^2 d\mu dx \nonumber\\
         &+ \frac12 \int_{-1}^0 |\mu|  (\psi^B(0,\cdot) - T_b^4)^2 d\mu \le  \frac12 \int_0^1 \mu (\psi_b-T_b^4)^2 d\mu,
    \end{align}
    holds for any $\alpha \in [0,1),$
where $\int_{-1}^1 \mu (\psi^B(x,\cdot)-(T^B(x))^4)^2 d\mu \ge 0$ for any $x\in [0,B]$.
\end{theorem}

The above existence theorem is proved by constructing a monotonic sequence of functions that converges to the solution of system \eqref{eq:B1}-\eqref{eq:B2}.
Before we construct the sequence, we first show the monotonicity property of equation \eqref{eq:B1}.

\subsection{Monotonicity of the second order ODE}\label{subMonot}
The monotonicity property of the second order ODE \eqref{eq:B1} is given in the following lemma.
\begin{lemma}\label{lm:T}
Given $0\le T_b\le \gamma$ for some constant $\gamma>0$. Let  $\phi$ be a strictly increasing bounded function on $\mathbb{R}$ and $g \in C([0,B])$ be a continuous bounded function satisfying $0\le g \le \phi(\gamma)$. Then there exists a unique bounded solution $T \in C^2([0,B])$ to the equation
    \begin{align}\label{eq:p2Tg}
        -\partial_x^2 T(x) +\phi(T(x)) = g(x), \quad \text{for any }x\in [0,B],
    \end{align}
    supplemented to the following boundary conditions 
    \begin{align*}
        T(0) = T_b, \quad \partial_x T(B)=0,
    \end{align*}
 and the solution satisfies $0\le T(x)\le \gamma$ for any $x\in[0,B]$. Moreover, let $T_1$, $T_2$ be two solutions to the above equation with source terms $g_1$, $g_2$ and boundary data $T_{b1}$, $T_{b2}$, respectively.  If $0\le g_1(x) \le g_2(x) \le \phi(\gamma)$ for all $x\in [0,B]$ and $0\le T_{b1} \le T_{b2}\le \gamma$, then $0\le T_1(x) \le T_2(x)\le \gamma$ for all $x\in [0,B]$.

 \end{lemma}
\begin{proof}
    \emph{{Existence}}.
    Let $h \in C[0,B]$, then the solution to the problem 
    \begin{align*}
        \left\{\begin{array}{ll}
        - \partial_x^2 f = h, \quad \text{on }[0,B],\\
         f(0) = 0, \quad \partial_x f(B)= 0
        \end{array}\right.
    \end{align*}
    can be explicitly written as 
    \begin{align}\label{ExF}
        f(x) = \int_0^x \int_t^B h(s) dsdt.
    \end{align}
    Denote the above mapping by $f=\mathcal{G}h$. Then $\mathcal{G}$ is a linear operator  from $C([0,B])$ to $C^2([0,B])$. We next show the mapping $\mathcal{N}$ defined by 
    \[
    \mathcal{N} T:=\mathcal{G}(g-\phi(T))+T_b\]
     has a fixed point. To show this, we recall the statement of the Leray-Schauder Theorem in appendix \ref{LSTheorem},  \cite[Theorem 11.3]{gilbarg2015elliptic}.
     
\begin{itemize}
\item  We consider the Banach space  $D=C([0,B])$ equipped with the supremum norm 
     \[
     \|f\|_D:=\sup_{x\in [0,B]}|f(x)|.
     \]
Since $\mathcal{N}: C([0,B]) \to C^2([0,B])$ and $C^2([0,B])$ is compactly embedded in $C([0,B])$, where $C^2([0,B])$ is equipped with the norm 
\[
\|f\|_{C^2([0,B])} := \sup_{x\neq y}\dfrac{|f(x)-f(y)|}{|x-y|^2} + \|f\|_{C([0,B])},
\]
  $\mathcal{N}$ is a compact map from $C([0,B])$ to itself. 
  
  \item We now need to show that the set 
 \[
 \mathcal{A}:=\Big\{T \in C([0,B]):T= \sigma \mathcal{N}T,\,\,\, \text{for some }  \sigma \in [0,1] \Big\},
 \]
is bounded. To show this, suppose $T=\sigma \mathcal{N}T$, then 
    \begin{align} \label{27}
        T = \sigma(\mathcal{G}(g-\phi(T)) + T_b),
    \end{align}
    which satisfies 
    \begin{align}
       - \partial_x^2 T +  \sigma \phi(T(x)) = \sigma g(x),\, x\in [0,B]\quad\text{ and } T(0) = \sigma T_b,\, \partial_x T(B)=0.
    \end{align}
    We first show $T \ge 0$. We use a contradiction argument. Suppose the minimum of $T$ is less than zero. Then let $x_m \in [0,B]$ be the minimum point, $T(x_m)<0$.
    By the monotonicity of $\phi$, $\phi(T(x_m)) < \phi(0)<0$. Hence 
    \begin{align*}
        -\partial_x^2 T(x_m) = \sigma (g(x_m) - \phi(x_m)) >  0,
    \end{align*} 
    i.e. around $x_m$, $T(x)$ is concave, which implies $x_m$ cannot be a local minimum point. Hence $x_m \not\in (0,B).$ If $x_m=B$, due to $ -\partial_x^2 T(B)>0$ and $\partial_x T(B)=0$, $x_m=B$ should be a  maximum point rather than a minimum point. Therefore, $x_m \neq B$. Otherwise if $x_m =0$, $T(0)=\sigma T_b \ge 0$, which contradicts to the assumption $T(x_m)<0$, so $x_m \neq 0$. Therefore, $x_m \not\in [0,B]$ and thus the set of minimum points with negative value is empty.
    Therefore, we conclude that $T(x) \ge 0$ for all $x\in [0,B].$

We can follow a similar contradiction argument to show $T \le \gamma$. Suppose $T$ attains its maximum at $x=x_M$ with $T(x_M)>\gamma$. Due to monotonicity of $\phi$, $\phi(T(x_M)) > \phi(\gamma)$ and thus
   \[-\partial_x^2 T(x_M) = \sigma (g(x_M)-\phi(T(x_M))) < \sigma (\phi(\gamma)-\phi(\gamma)) < 0,\]
   which implies $x_M$ is a loal minimum. Hence $x_M \not\in (0,B)$. If $x_M=0$, then $T(x_M)=\sigma T_b \le \gamma$, so $x_M \neq 0$. If $x_M=B$, due to $\partial_x T(B)=0$, $\partial_x^2 T(B) \ge 0$ also implies that $B$ is a minimum point, which contradicts the assumption that $x_M=B$ is a maximum point. Therefore, we conclude that the set of maximum point with values bigger than $\gamma$ is empty. Therefore, $T(x_M) \le \gamma$ and so $T(x) \le \gamma$ for all$x\in [0,B]$. Finally, from  $0\le T \le \gamma$, we conclude that
\[
0 \le \|T\|_{C([0,B])} \le \gamma\quad \text{for any}\,\,\,\sigma \in [0,1].
\]
We can thus conclude that the set $\mathcal{A}$ is bounded for any $\sigma \in [0,1]$.
   \end{itemize}
  Consequently, the hypotheses of Leray-Schauder fixed point theorem are verified, which implies that $\mathcal{N}$ has a fixed point in $C([0,B])$. Therefore, there exists a solution in $C([0,B])$ for equation \eqref{eq:p2Tg}. Moreover, from \eqref{27} and the formula \eqref{ExF}, $T \in C^2([0,B])$.
    
        \emph{{Uniqueness}.} Suppose $T_1$ and $T_2$ are two distinct solutions to \eqref{eq:p2Tg}, then $T_1-T_2$ satisfies
    \begin{align*}
        \partial_x^2(T_1-T_2) - (\phi(T_1)-\phi(T_2)) = 0,
    \end{align*}
    with 
    \begin{align*}
        (T_1-T_2)(0)= 0,\, \partial_x (T_1-T_2)(B)=0.
    \end{align*}
We first show any maximum point $x_M \in [0,B]$ of $T_1(x)-T_2(x)$ cannot be positive. Otherwise, $T_1(x_M)-T_2(x_M) > 0$ and then by the monotonicity of $\phi$, $\phi(T_1(x_M))>\phi(T_2(x_M))>0$, hence $\partial_x^2(T_1-T_2)(x_M) = \phi(T_1(x_M))-\phi(T_2(x_M))>0$, which implies $T_1-T_2$ is convex near $x_M$. Thus $x_M$ can only be a local minimum or on the boundary. By the assumption, $x_M$ is a maximum point, $x_M \not\in (0,B)$. If $x_M=0$, $T_1(0)-T_2(0)= 0$ contradicts the assumption $(T_1-T_2)(x_M)>0$, so $x_M \neq 0$. If $x_M=B$, then $\partial_x (T_1-T_2)(B)=0$ and $\partial_x^2(T_1-T_2)(B)=0$ implies $B$ is a local minimum and thus cannot be a maximum point. Hence the set of $\{x:(T_1-T_2)(x)>0\}$ is empty.

    Similarly, we can show any minimum point $x_m \in [0,B]$ of $T_1(x)-T_2(x)$ cannot have values $T_1(x_m)-T_2(x_m) <0$.  Otherwise, $T_1(x_m)-T_2(x_m) <0$ and $\partial_x^2 (T_1-T_2)(x_m) = \phi(T_1(x_m))-\phi(T_2(x_m))<0$ and $T_1-T_2$ is locally concave at $x_m$. So $x_m$ cannot be a local minimum, $x_m \not\in (0,B)$. The fact that $(T_1-T_2)(0)=0$ implies $x_m \neq 0$. The condition $\partial_x (T_1-T_2)(B)=0$ and $\partial_x^2 (T_1-T_2)(B) <0$ implies that $B$ is a local maxium point and thus cannot be a minimum point, i.e. $x_m \neq B$. Therefore, we conclude that  the set of $\{x:(T_1-T_2)(x)<0\}$ is empty.

    Combining the above arguments, the set $\{x: (T_1-T_2)(x)\neq 0\}$ is empty. That is, $T_1=T_2$ on $[0,B]$ and the solution is unique.


    \emph{{Monotonicity}.} Assume $g_1(x)\le g_2(x)$ on $x\in [0,B]$ and $T_{b1} \le T_{b2}$, and $T_1$, $T_2$ are two solutions with source terms $g_1,\,g_2$ and boundary data $T_{b1},T_{b2}$, respectively. Suppose the maximum of $T_1(x)-T_2(x)$ occurs at $x=x_M \in (0,B)$ with $T_1(x)-T_2(x) > 0$. Then $\partial_x^2 (T_1-T_2)(x_M) = \phi(T_1(x_M)) - \phi(T_2(x_M)) > 0$ and so $x=x_M$ can not be a local maximum, $x_M \not\in (0,B)$.
     If the maximum occurs at $x=B$ and $(T_1-T_2)(B)>0$, then $\partial_x^2 (T_1-T_2)(B) > 0$ and so $x=B$ is a minimum point due to $\partial_x T_1(B)=\partial_x T_2(B)$, which contradicts the assumption that $x=B$ is a maximum point. If the maximum occurs at $0$, we have $T_1(0)-T_2(0)=T_{b1}-T_{b2}\le 0$, which contradicts the assumption $T_1(x_M)-T_2(x_M)>0$.
     Therefore, we have $T_1 \le T_2$ on $[0,B]$.
\end{proof}

\subsection{Existence and uniqueness}\label{SubExUni}
With the monotonicity of solutions to the second order ODE in Lemma \ref{lm:T} and the monotonicity of solutions for the linear transport equation in Lemma \ref{lm:psi} in Appendix A, we are prepared to prove the existence for system \eqref{eq:B1}-\eqref{eq:B2}. Before we give the proof, we first introduce the concept of subsolutions.




\begin{definition} \label{def:sub}
We call $(T,\psi) \in C^2([0,B]) \times C^1([0,B]\times [-1,1])$ 
 a subsolution of \eqref{eq:B1}-\eqref{eq:B2} if the following inequalities hold 
\begin{align}
  -  \partial_{x}^2 T - \langle \psi - T^4 \rangle  \le~& 0,\label{eq:m1sub}\\
    \mu \partial_x \psi + (\psi-T^4) \le ~&0,\label{eq:m2sub}
\end{align}
and on the boundary they satisfy $0 \le T(0) \le T_b$, $0\le \psi(0,\mu) \le \psi_b(\mu)$ for $\mu \in (0,1]$.
\end{definition}

\begin{proof}[Proof of Theorem \ref{thm.bd}, existence and uniqueness]
    The proof is divided into four steps. An approximate sequence is constructed in the first step. Then, we prove the existence of solutions for system \eqref{eq:B1}-\eqref{eq:B2}. After that, the uniqueness of solutions is shown using a contradiction argument. Finally, a weighted energy estimate is derived.
    
    \smallskip

    \emph{{Step 1:  Construction of an approximate sequence}.} 
    Let $\phi(T)$ be a strictly increasing function defined by 
    \begin{align*}
        \phi(T) = \left\{\begin{array}{cl}
            \frac{T}{1-T}&\text{if } T<0,\\
            T^4,&\text{if } 0\le T \le \gamma,\\
            \gamma^4 + \frac{T-\gamma}{1+T-\gamma},&\text{if } T>\gamma.
        \end{array}\right.
    \end{align*}
    The above function satisfies  $\phi(T)<0$ if $T<0$ and $\phi(T)> \gamma^4$ if $T> \gamma$.

    We construct a sequence of functions $\{T^k=T^k(x),\,\psi^k=\psi^k(x,\mu)\}_{k=0}^\infty$ starting with $T^0$, $\psi^0$ being a subsolution given by Definition \ref{def:sub}\footnote{
    We may take $T^0=0$ and $\psi^0=0$, which is  a subsolution. 
}, i.e.
    \begin{align}
        &-\partial_x^2T^0 +2 \phi (T^0) \le  \langle \psi^0\rangle,\label{eq:01}\\
        &\mu \partial_x \psi^0 + \psi^0 \le (T^0)^4.\label{eq:02}
    \end{align}
    with 
    \begin{align*}
        &0\le T^0(0)\le T_b,\quad \partial_x T^0(B)=0,\\
        &0\le \psi^0(0,\mu) \le \psi_b(\mu),\quad \psi(B,\mu)=\psi(B,-\mu),\quad \text{for any }\mu \in(0,1].
    \end{align*}
    Given $(T^k,\psi^k)$, $(T^{k+1},\psi^{k+1})$ is obtained by solving 
   \begin{align}
        &-\partial_x^2 T^{k+1} + 2\phi(T^{k+1}) = \langle \psi^{k}\rangle, \label{eq:0k+1}\\
        &\mu \partial_x \psi^{k+1} + \psi^{k+1} = (T^k)^4, \label{eq:1k+1}
    \end{align}
    with boundary conditions
    \begin{align*}
        &T^{k+1}(0) = T_b,\quad \partial_x T^{k+1}(B)=0,\\
        &\psi^{k+1}(0,\mu) = \psi_b, \quad \psi^{k+1}(B,\mu) = \psi^{k+1}(B,-\mu),\quad \text{for any }\mu>0.
    \end{align*}

    \smallskip

    \emph{{Step 2:  Existence.}} 
    Assume $(T^k,\psi^k) \in  C^2([0,B])\times C^1([0,B]\times [-1,1])$ satisfy $0\le T^k \le \gamma$ and $0 \le \psi^k \le \gamma^4$, then  Lemma \ref{lm:T} and Lemma \ref{lm:psi} imply that there exists a unique solution $(T^{k+1},\psi^{k+1}) \in C^2([0,B])\times C^1([0,B]\times [-1,1])$ to system \eqref{eq:0k+1}-\eqref{eq:1k+1} and the solution satisfies $0 \le T^{k+1} \le \gamma$ and $0\le \psi^{k+1} \le \gamma^4$. By induction, we conclude that for any $k\ge 1$ and $k \in \mathbb{N}$,
    \begin{align}\label{eq:216}
        0\le T^k(x) \le \gamma,\quad 0\le \psi^k(x,\mu) \le \gamma^4, \text{ for any }x\in [0,B],\, \mu \in [-1,1].
    \end{align}

    Next we prove by induction that $T^{k+1}(x) \ge T^k(x),\psi^{k+1}(x,\mu)\ge \psi^k(x,\mu)$ for any $x\in [0,B]$ and $\mu \in [-1,1]$.
    First $T^1,\psi^1$ solves 
    \begin{align*}
        &-\partial_x^2 T^1 + 2 \phi(T^1) = \langle \psi^0 \rangle,\\
        &\mu \partial_x \psi^1 + \psi^1 = (T^0)^4.
    \end{align*}
    Comparing this with  equations \eqref{eq:01}-\eqref{eq:02} for $(T^0,\psi^0)$, we get from Lemma \ref{lm:T} and Lemma \ref{lm:psi} that $T^1(x) \ge T^0(x)$ and $\psi^1(x,\mu) \ge \psi^0(x,\mu)$ for $x\in [0,B]$ and $\mu \in [-1,1]$.
    
    Suppose $T^k(x) \ge T^{k-1}(x)$ and $\psi^k(x,\mu)\ge \psi^{k-1}(x,\mu)$, then 
    $$ - \partial_x^2 T^{k+1} + 2\phi(T^{k+1}) = \langle \psi^{k}\rangle,$$ 
    $$ - \partial_x^2 T^{k} + 2\phi(T^{k}) = \langle \psi^{k-1}\rangle,$$
    and by Lemma \ref{lm:T}, $T^{k+1}(x) \ge T^k(x)$ for any $x\in [0,B]$. 
    Similarly, we have 
    \begin{align*}
        \mu \partial_x \psi^{k+1} + \psi^{k+1} = (T^k)^4,\\
        \mu \partial_x \psi^{k} + \psi^k = (T^{k-1})^4,
    \end{align*}
    and Lemma \ref{lm:psi} implies that $\psi^{k+1}(x,\mu) \ge \psi^k(x,\mu)$ for any $x\in [0,B]$ and $\mu \in [-1,1]$. Therefore, we conclude that for any $k\ge 1$ and $k \in \mathbb{N}$,
    \begin{align}
        T^{k+1}(x) \ge T^k(x),\quad \psi^{k+1} \ge \psi^k(x), \quad \text{for any } x\in [0,B],\, \mu \in [-1,1].
    \end{align}
Combining this with \eqref{eq:216}, we conclude that
     $\{T^k\}_{k=0}^\infty,\{\psi^k\}_{k=0}^\infty$ are bounded increasing sequences. Hence there exists a converging subsequence such that 
    \begin{align*}
        &\lim_{k\to\infty} T^k (x)= T^B(x), \quad\lim_{k\to \infty} \psi^k(x,\mu) = \psi^B(x,\mu), 
    \end{align*}
    for any $x\in [0,B], \mu \in [-1,1]$. Moreover, by the Beppo Levi's lemma, it holds that 
    \begin{align*}
   \lim_{k\to\infty} \int_{0}^x T_k(y) dy =\int_{0}^x T^B(y) dy, \quad \lim_{k\to\infty} \int_{0}^x \psi_k(y,\mu) dy =\int_{0}^x \psi^B(y,\mu) dy.
    \end{align*}    
    Hence, we can pass to the limit $k\to\infty$ in
    \begin{align}
    	T^{k+1} = \mathcal{G}\left(\langle \psi^k\rangle - \phi(T^{k+1})\right) + T_b ,
    \end{align}
    and get 
    \begin{align}
    T^B =   \lim_{k\to\infty}\mathcal{G}\left(\langle \psi^k\rangle  - \phi(T^{k})\right) + T_b = \mathcal{G}\left(\langle \psi\rangle - \phi(T)\right) + T_b
    \end{align}
    and by the continuity of the operator $\mathcal{G}$, $T^B \in C^2([0,B]$. In a similar way we can also get $\psi\in C^1([0,B]\times [-1,1])$. Hence, we can apply the Dini's theorem and conclude that the convergence of $(T_k,\psi_k)$ to $(T^B,\psi^B)$ is uniform in $C^2([0,B])\times C^1([0,B]\times [-1,1])$.
    
Moreover, since $0\le T_k,\psi_k \le \gamma$, $(T^B,\psi^B)$ satisfies $0\le T\le \gamma$, $0\le \psi\le \gamma^4$. We can pass to the limit in \eqref{eq:0k+1}-\eqref{eq:1k+1} and get that $(T^B,\psi^B)$ satisfies the system
    \begin{align*}
        &-\partial_x^2 T^B + 2\phi(T^B) = \langle \psi^B \rangle,\\
        &\mu\partial_x \psi^B+\psi^B = (T^B)^4,
    \end{align*}
    with 
    \begin{align*}
        &T^B(0)=T_b,\quad \partial_x T^B(B)=0,\\
        &\psi^B(0,\mu) = \psi_b,\quad \psi^B(B,\mu)=\psi(B,-\mu),\quad \text{for any } \mu>0.
    \end{align*}
    Since $0\le T^B\le \gamma$, $\phi^B(T)=T^{4}$ and the above system is equivalent to \eqref{eq:B1}-\eqref{eq:B2}. Therefore, there exists a solution $(T^B,\psi^B) \in C^2([0,B])\times C^1([0,B]\times [-1,1])$ to system \eqref{eq:B1}-\eqref{eq:B2} with boundary conditions \eqref{eq:B1b}-\eqref{eq:B2b}.

    \smallskip

    \emph{{Step 3: Uniqueness.}}
Suppose $(T_1^B,\psi_1^B),(T_2^B,\psi_2^B)$ are two solutions to \eqref{eq:B1}-\eqref{eq:B2} with boundary conditions \eqref{eq:B1b}-\eqref{eq:B2b}. Define 
    \begin{align*}
        T^0 = \max(T_1^B,T_2^B),\quad \psi^0=\max(\psi_1^B,\psi_2^B),
    \end{align*}
    which satisfies \eqref{eq:01}-\eqref{eq:02} and is thus a subsolution.
    We construct a sequence of functions $\{T^k(x),\psi^k(x,\mu)\}_{k=0}^\infty$ by solving iteratively
    \begin{align*}
        &-\partial_x^2 T^{k+1} + 2\phi(T^{k+1}) = \langle \psi^{k}\rangle,\\
        &\mu \partial_x \psi^{k+1} + \psi^{k+1} = (T^k)^4.
    \end{align*}
    Then Lemma \ref{lm:psi} implies that there exists a subsequence such that  $(T^{k},\psi^{k})$ converges to a solution $(T^B,\psi^B)$ of system \eqref{eq:m1}-\eqref{eq:m2} as $k\to \infty$ and $T^B\ge T^0,\psi^B \ge \psi^0$. Hence $g^B= T^B-T_1^B,\varphi^B = \psi^B-\psi_1^B$ is a nonnegative solution to the following system 
    \begin{align}
        &- \partial_x^2 g^B + 2((T^B)^4 - (T_1^B)^4) = \langle \varphi\rangle,\label{eq:g/1}\\
        &\mu \partial_x \varphi^B + \varphi^B = ((T^B)^4 - (T_1^B)^4),\label{eq:g/2}
    \end{align}
    with 
    \begin{align*}
        &g^B(0)=0,\quad \partial_x g^B(B)=0,\\
        &\varphi^B(0,\mu) = 0,\quad \varphi^B(B,\mu)=\varphi^B(B,-\mu),\text{ for any }\mu>0.
    \end{align*}
    Comparing \eqref{eq:g/1} to \eqref{eq:g/2} gives
    \begin{align*}
        \partial_x( \partial_x g^B - \langle \mu \varphi^B\rangle) =0.
    \end{align*}
    Due to the boundary conditions $\partial_x g^B(B)=0$ and $\langle \mu \varphi^B(B,\cdot)\rangle = 0$,
    \begin{align*}
        \partial_x g^B (x) = \langle \mu \varphi^B(x,\cdot) \rangle, \text{ for any }x\in [0,B].
    \end{align*}
    Using \eqref{eq:sol/mu-}, we have
    \begin{align*}
        \varphi^B(0,\mu) =& -\int_0^B \frac{1}{\mu} e^{\frac{2B-s}{\mu}} ((T^B)^4-(T_1^B)^4)(s) ds \nonumber\\
        &- \int_x^B \frac{1}{\mu} e^{-\frac{x-s}{\mu}} ((T^B)^4-(T_1^B)^4)(s) ds, \text{ for }\mu<0.
    \end{align*}
    If $T^B(x)> T_1^B(x)$ on some interval $(a,b)\subset [0,B]$, then the above equation implies that $\varphi^b(0,\mu)> 0$ for $\mu <0$. This, together with the boundary condition $\varphi^B(0,\mu)=0$ for $\mu>0$, implies 
    \begin{align}
        \partial_x g^B(0) = \langle \mu \varphi^B(0,\cdot)\rangle = \int_{-1}^0 \mu \varphi^B(0,\mu) d\mu < 0.
    \end{align}
    This contradicts to the fact that $g^B(x)\ge 0$ for any $x\in [0,B]$ and $g^B(0)=0$. Hence $T^B(x)=T_1^B(x)$ for $x\in [0,B]$ almost everywhere.
    
    Due to $T^B=T_1^B$,  equation \eqref{eq:g/2} becomes $\mu\partial_x \varphi^B +\varphi^B=0$, with boundary conditions $\varphi^B(0,\mu)=0, \varphi(B,\mu)=\varphi(B,-\mu)$ for $\mu>0$. Therefore, $\psi^B(x,\mu)=\psi_1^B(x,\mu)$ for $x\in [0,B], \mu \in [-1,1]$.
    By the same argument we also obtain $T^B=T_2^B, \psi^B=\psi_2^B$. Hence $T_1^B=T_2^B$ and $\psi_1^B=\psi_2^B$,  i.e. the solution to system \eqref{eq:B1}-\eqref{eq:B2b} is unique.

\end{proof}

    \subsection{Weighted estimate} \label{sec:weB}
    Next we derive the weighted estimate \eqref{eq:estB2}.
\begin{proof}[Proof of Theorem \ref{thm.bd}, weighted estimate]
    We multiply \eqref{eq:B1} by $(T^B)^4$, integrate over $[0,x]$ and multiply \eqref{eq:B2} by $\psi^B$, integrate over $(z,\mu)\in [0,x]\times[-1,1]$, and combine the results to get 
    \begin{align}\label{lm4.0}
        -\int_0^x (T^B)^4 \partial_z^2 T^B dz + \frac12 \int_0^x \int_{-1}^1 \mu \partial_z (\psi^B)^2 d\mu dz + \int_0^x\int_{-1}^1 (\psi^B - (T^B)^4)^2 d\mu dz  = 0.
    \end{align}
    Using integration-by-parts,
      we obtain 
    \begin{align}\label{lm4.1}
        -\int_0^x(T^B)^4 \partial_z^2 T^B dz &= \int_0^x 4 (T^B)^3 |\partial_z T^B|^2 dz - (T^B)^4(z) \partial_z T^B(z) \bigg|_{z=0}^x .
    \end{align}
    and 
    \begin{align}\label{lm4.2}
        \int_0^x \int_{-1}^1 \mu \partial_z (\psi^B)^2 d\mu dz = \int_{-1}^1 \mu (\psi^B)^2(z,\cdot) d\mu \bigg|_{z=0}^x. 
    \end{align}
    Comparing  equation \eqref{eq:B1} to \eqref{eq:B2} gives 
    \begin{align}\label{lm4.3}
        \partial_x^2 T^B = \partial_x \langle \mu \psi^B \rangle,
    \end{align}
    this together with $\partial_x T^B (B)= \langle \mu \psi^B (B,\cdot) \rangle = 0$ implies 
    \begin{align}\label{lm4.=}
        \partial_x T^B(x) = \langle \mu \psi^B (x,\cdot)\rangle,\quad \text{for any } x\in [0,B].
    \end{align}
    Using this and taking \eqref{lm4.1} and \eqref{lm4.2} into \eqref{lm4.0}, we get 
    \begin{align*}
        &\int_0^x 4 (T^B)^3 |\partial_z T^B|^2 dz + \int_0^x \int_{-1}^1 (\psi^B - (T^B)^4)^2 d\mu dz \\
        &\quad  + \frac12 \int_{-1}^1 \mu (\psi^B(x,\cdot) - (T^B(x))^4)^2 d\mu - \frac12 \int_{-1}^1 \mu (\psi^B(0,\cdot) - T_b^4)^2 d\mu = 0.
    \end{align*}
    The boundary condition \eqref{eq:B2b} implies 
    \begin{align}\label{lm4.eqb}
        \int_{-1}^1 \mu (\psi^B(0,\cdot) - T_b^4)^2 d\mu = \int_{-1}^0 \mu (\psi^B(0,\cdot) - T_b^4)^2 d\mu + \int_0^1 \mu (\psi_b-T_b^4)^2 d\mu.
    \end{align}
    Taking it into the previous equation leads to the estimate \eqref{eq:estB2} with $\alpha=0$:
    \begin{align}\label{eq:estB1}
        &\int_0^x 4 (T^B)^3 |\partial_z T^B|^2 dx + \int_0^x \int_{-1}^1 (\psi^B - (T^B)^4)^2 d\mu dz  \nonumber\\
        &\quad + \frac12 \int_{-1}^1 \mu (\psi^B(x,\cdot) - (T^B(x))^4)^2 d\mu  + \frac12 \int_{-1}^0 |\mu|  (\psi^B(0,\cdot) - T_b^4)^2 d\mu  \nonumber\\
        &\qquad =\frac12 \int_0^1 \mu (\psi_b-T_b^4)^2 d\mu.
    \end{align}
    Note here we have 
    \begin{align*}
        \int_{-1}^1 \mu (\psi^B(x,\cdot) - (T^B(x))^4)^2 d\mu \ge 0
    \end{align*}
    for any $x \in [0,B]$, which is due to the fact that $\int_{-1}^1 \mu (\psi^B(x,\cdot) - (T^B(x))^4)^2 d\mu$ is a non-increasing function of $x$ and 
    \begin{align}
        \int_{-1}^1 \mu (\psi^B(B,\cdot) - (T^B(B))^4)^2 d\mu  =0,
    \end{align}
    which is due to the boundary condition $\psi^B(B,\mu)=\psi^B(B,-\mu)$. The non-increasing fact of $\int_{-1}^1 \mu (\psi^B(x,\cdot) - (T^B(x))^4)^2 d\mu$ can be seen if we integrate over $[x_1,x_2]$ with $0\le x_1 < x_2 \le B$ in the derivation above, the inequality
    \begin{align*}
        &\frac12 \int_{-1}^1 \mu (\psi^B(x,\cdot) - (T^B(x))^4)^2 d\mu \bigg|_{x_1}^{x_2} \\
        &\quad = -\int_{x_1}^{x_2} 4 (T^B)^3 |\partial_x T^B|^2 dx - \int_{x_1}^{x_2} \int_{-1}^1 (\psi^B - (T^B)^4)^2 d\mu dx \le 0
    \end{align*}
    is satisfied.


Next, let $\alpha \in (0,1)$ be a constant.
We multiply \eqref{eq:B1} by $e^{2\alpha x} (T^B)^4$, and \eqref{eq:B2} by $e^{2\alpha x} \psi^B$, and integrate to get 
    \begin{align}\label{lm4.4}
        &- \int_0^x e^{2\alpha z} (T^B)^4 \partial_z^2 T^B dx + \frac12 \int_0^x \int_{-1}^1 \mu e^{2\alpha z} \partial_z (\psi^B)^2 d\mu dx \nonumber\\
        &\quad + \int_0^x\int_{-1}^1 e^{2\alpha z} (\psi^B - (T^B)^4)^2 d\mu dx = 0. 
    \end{align}
    Using integration-by-parts, we get 
    \begin{align*}
        - \int_0^x e^{2\alpha z} (T^B)^4 \partial_z^2 T^B dz & = \int_0^x e^{2\alpha z} 4 (T^B)^3 |\partial_x T^B|^2 dx + 2\alpha \int_0^x e^{2\alpha z} (T^B)^4 \partial_z T^B dx \nonumber \\
        &\quad - e^{2\alpha z} (T^B)^4(z) \partial_z T^B (z)\bigg|_{z=0}^x, \nonumber
    \end{align*}
    and 
    \begin{align*}
        \int_0^x \int_{-1}^1 \mu e^{2\alpha z} \partial_z (\psi^B)^2 d\mu dz &= - 2\alpha \int_0^x \int_{-1}^2 \mu e^{2\alpha z} (\psi^B)^2 d\mu dz + \int_{-1}^1 \mu e^{2\alpha z} (\psi^B)^2(z,\cdot) d\mu \bigg|_0^x. \nonumber
    \end{align*}
    Taking the above equations into \eqref{lm4.4} and using the relation \eqref{lm4.=} gives
    \begin{align*}
        &\int_0^x e^{2\alpha z} 4 (T^B)^3 |\partial_z T^B|^2 dx + \int_0^x\int_{-1}^1 e^{2\alpha z} (\psi^B - (T^B)^4)^2 d\mu dz \nonumber\\
        &\quad - \alpha \int_0^x \int_{-1}^1 \mu e^{2\alpha z} (\psi^B - (T^B)^4)^2 d\mu dz = -\frac12 \int_{-1}^1 \mu e^{2\alpha z} (\psi^B(z,\cdot)-(T^B(z))^4)^2 d\mu \bigg|_0^x .
    \end{align*}
    Using the relation \eqref{lm4.eqb} and $1-\alpha \mu \ge  1-\alpha$ for $0\le \alpha<1$ in the above equation leads to
    \begin{align}
        &\int_0^x e^{2\alpha z} 4 (T^B)^3 |\partial_z T^B|^2 dx + (1-\alpha)\int_0^x\int_{-1}^1 e^{2\alpha z} (\psi^B - (T^B)^4)^2 d\mu dz \nonumber\\
        &\quad + \frac12 \mu\int_{-1}^1 e^{2\alpha x} (\psi^B(x,\cdot)-(T^B(x))^4)^2 d\mu + \frac12 \int_{-1}^0 |\mu| (\psi^B(0,\cdot)-T_{b})^4)^2 d\mu \nonumber\\
        &\qquad \le \frac12 \int_0^1 \mu (\psi_b - T_b^4)^2 d\mu,
    \end{align}

    and finishes the proof of Theorem \ref{thm.bd}.
\end{proof}

\section{Existence and decay results for the nonlinear Milne problem}\label{proofTh1}

In this section, we consider the existence of the nonlinear Milne problem and prove Theorem \ref{thm:ex}.
We first pass to the limit $B\to\infty$ in system \eqref{eq:B1}-\eqref{eq:B2} and show the existence of weak solutions to system \eqref{eq:m1}-\eqref{eq:m2} as well as the weighted estimate \eqref{eq:es-2}. Then the exponentially decay of solutions is shown.
\subsection{Existence}\label{ProofEx1}
We first pass to the limit $B\to\infty$ on system \eqref{eq:B1}-\eqref{eq:B2} and show the existence of weak solutions for system \eqref{eq:m1}-\eqref{eq:m2}.


\begin{proof}[Proof of Theorem \ref{thm:ex}, existence]
First we show the convergence of $T^B(x)$ as $B\to\infty$. To show this, we first prove that there exists a subsequence $\{T^{B_k}(B_k)\}$ converging to some constant $T_\infty$.
From the weighted estimate \eqref{eq:estB2} (taking $x=B$ in \eqref{eq:estB2}), we can get 
\begin{align*}
    &(T^B)^{5/2}(B) - T_b^{5/2} = \int_0^B \partial_x (T^B)^{5/2} dx = \int_0^B e^{-\alpha x} e^{\alpha x} \partial_x (T^B)^{5/2} dx \nonumber\\
    &\quad \le \left(\int_0^B e^{2\alpha x} |\partial_x (T^B)^{5/2}|^2 dx\right)^{1/2}  \left(\int_0^B e^{-2\alpha x} dx\right)^{1/2} \nonumber\\
    &\quad \le \frac{1}{(2\alpha)^{1/2}} \sqrt{(1-e^{-2\alpha B})} \left(\frac{25}{16}\int_0^B e^{2\alpha x} 4(T^B)^3|\partial_x (T^B)|^2 dx\right)^{1/2} \nonumber\\
    &\quad \le \frac{5}{4(2\alpha)^{1/2}} \sqrt{(1-e^{-2\alpha B})} \left(\frac12 \int_0^1 \mu (\psi_b-T_b^4)^2 d\mu\right)^{1/2} \nonumber\\
    &\quad \le \frac{5}{4(2\alpha)^{1/2}} \left(\frac12 \int_0^1 \mu (\psi_b-T_b^4)^2 d\mu\right)^{1/2}.
\end{align*}
Hence $(T^B)^{5/2}(B)$ is bounded for any $B>0$. By the Bolzano–Weierstrass theorem, there exists a subsequence $\{B_k\}$ such that 
\begin{align*}
    T^{B_k}(B_k) \to T_\infty, \quad \text{as } k\to \infty,
\end{align*}
where $T_\infty$ is some constant. Since $0\le T^k(B_k)(B_k) \le \gamma$ according to Theorem \ref{thm.bd}, $0\le T_\infty \le \gamma$. 

Next we show $T^B(x)-T^B(B)$ is uniformly bounded in $ L^2([0,B])$.  
From estimate \eqref{eq:estB2}, it holds the inequality
\begin{align*}
    (1-\alpha)\int_0^B\int_{-1}^1 e^{2\alpha x} (\psi^B - (T^B)^4)^2 d\mu dx \le  \frac12 \int_0^1 \mu (\psi_b-T_b^4)^2 d\mu.
\end{align*}
Using this estimate and the relation 
\begin{align*}
    \left(\int_{-1}^1 \mu \psi^B d\mu\right)^2 &=  \left(\int_{-1}^1 \mu (\psi^B - (T^B)^4) d\mu\right)^2 \le \int_{-1}^1 \mu^2 d\mu \int_{-1}^1 (\psi^B - (T^B)^4)^2 d\mu \nonumber\\
    &\le \frac{2}{3} \int_{-1}^1 (\psi^B - (T^B)^4)^2 d\mu,
\end{align*}
we can get 
\begin{align*}
    \int_{0}^B e^{2\alpha x} |\langle \mu \psi^B \rangle|^2 dx &\le \frac23\int_0^B\int_{-1}^1 e^{2\alpha x} (\psi^B - (T^B)^4)^2 d\mu dx \nonumber\\
    &\le \frac13 \frac{1}{1-\alpha}  \int_0^1 \mu (\psi_b-T_b^4)^2 d\mu.
\end{align*}
Due to \eqref{lm4.=}, the above inequality implies 
\begin{align}\label{conv5}
    \int_{0}^B e^{2\alpha x} |\partial_x T^B |^2 dx \le \frac13 \frac{1}{1-\alpha}  \int_0^1 \mu (\psi_b-T_b^4)^2 d\mu.
\end{align}
Therefore,
\begin{align*}
   \int_0^B (T^B(x)- T^B(B))^2 dx &= \int_0^B \left(\int_x^B \partial_z T^B(z) dz \right)^2 dx \\
   &=  \int_0^B \left(\int_x^B e^{-\alpha z} e^{\alpha z} \partial_z T^B(z) dz \right)^2 dx  \\
   &\le \int_0^B\left( \int_x^B e^{2\alpha z} |\partial_z T^B|^2(z) dz \cdot \int_x^B e^{-2\alpha z} dz  \right)dx \\
   &\le \frac{1}{2\alpha}\int_0^B e^{2\alpha x} |\partial_x T^B|^2 dx \cdot 
   \int_0^B (e^{-2\alpha x}-e^{-2\alpha B}) dx \\
   &\le \frac16 \frac{1}{1-\alpha} \frac{e^{2\alpha B}-2\alpha B - 1}{2\alpha^2} e^{-2\alpha B} \int_0^1 \mu (\psi_b-T_b^4)^2 d\mu \\
   &\le  \frac16 \frac{1}{2\alpha^2(1-\alpha)}\int_0^1 \mu (\psi_b-T_b^4)^2 d\mu,
\end{align*}
where $(e^{2\alpha B}-2\alpha B - 1)e^{-2\alpha B} = 1- 2\alpha B e^{-2\alpha B} - e^{-2\alpha B} \le 1$.
Hence $T^B(x)-T^B(B) \in L^2_{\rm loc}(\mathbb{R}_{+})$ is uniformly bounded. 
Moreover, due to the boundary condition \eqref{eq:B1b}, $\partial_x T^B(B)=0$ and so the estimate \eqref{conv5} implies 
\begin{align}
    \int_0^B |\partial_xT^B(x) - \partial_x T^B(B)|^2 dx \le \frac{1}{3(1-\alpha)} \int_0^1 \mu (\psi_b-T_b^4)^2 d\mu,
\end{align}
thus $T^B(x) - T^B(B)$ is uniformly bounded in $H^1_{\rm loc}(\mathbb{R}_+)$. Moreover, $\|\partial_x^2 T^B\|_{L^2([0,B])} =  \|\langle \psi^B-(T^B)^4\rangle\|_{L^2([0,B])}\le \sqrt{2} \|\psi^B - (T^B)^4\|_{L^2([0,B]\times [-1,1])}$ is uniformly bounded. Therefore, there exists a subsequence $B_k$ such that 
\begin{align}\label{eq:h1conv}
    T^{B_k}(x) - T^{B_k}(B_k) \rightharpoonup T(x) - T_\infty,\quad \text{weakly in } H^2_{\rm loc}(\mathbb{R}_+). 
\end{align}
By the continuous embedding $  C_{\rm loc}^1 \subset H^2_{\rm loc}(\mathbb{R}_+)$, we get from the above convergence 
\begin{align}\label{eq:convT}
    T^{B_k}(x) \to T(x),\quad \text{strongly in } C^1_{\rm loc}(\mathbb{R}_+).
\end{align}

\smallskip

Next we show the convergence of $\psi^B(x)$. Due to the uniform estimate \eqref{eq:estB1}, $\|\psi^B - (T^B)^4 \|_{L^2([0,B]\times [-1,1])}$ is uniformly bounded in $B$. This, together with \eqref{eq:convT} implies that  there exists a subsequence $\{\psi^{B_k}\}$ such that as $k\to\infty$,
\begin{align}\label{eq:convpsi2}
    \psi^{B_k}(x,\mu) - (T^{B_k})^4(x) \rightharpoonup \psi(x,\mu) - T^4,\quad \text{ weakly in } L^2_{\rm loc}(\mathbb{R}_+\times [-1,1]).
\end{align}

\smallskip 

Next we show the limit $T,\psi$ solves system \eqref{eq:m1}-\eqref{eq:m2} with boundary conditions \eqref{eq:bd1}-\eqref{eq:bd2}.
First we show the limit $T(x)$, $\psi(x)$ satisfies the boundary condition \eqref{eq:bd1}-\eqref{eq:bd2}. Due to the trace operator $u \mapsto u(0)$ is continous (see \cite[Ex 8.18]{brezis2010functional}) for any $u \in H^1_{\rm loc}(\mathbb{R}_+)$, we can get from \eqref{eq:h1conv} that $T(0)=\lim_{k\to \infty} T^{B_k}(0)= T_b$. One can also check that $\varphi \to \varphi(0,\cdot)$ is also continous for any $\varphi \in L^2_{\loc}(\mathbb{R}_+\times [-1,1]) \cap \{\varphi: \mu \partial_x \varphi + \varphi \in C_{\rm loc}(\mathbb{R}_+)\}$ (one can use the formulas \eqref{eq:sol/mu+}-\eqref{eq:sol/mu-} to show this). Hence, $\psi^{B_k}(0,\mu) \rightharpoonup \psi^k(0,\mu)$ as $B_k \to \infty$.
To show $(T,\psi)$ satisfies \eqref{eq:m1}-\eqref{eq:m2}, we apply a test function $h=h(x) \in C^1([0,x])$, $\varphi \in C^1([0,x]\times [-1,1])$ to \eqref{eq:B1} and \eqref{eq:B2}, integrate by parts, and obtain 
\begin{align}
    &\int_{0}^x \partial_z T^B \partial_z h dz - \partial_z T^B(z) h(z)\bigg|_0^x + \int_0^x \int_{-1}^1 (\psi^B - (T^B)^4) h d\mu dz = 0,\label{eq:w11}\\
    &\int_0^x \int_{-1}^1 \mu \psi^B \partial_z \varphi d\mu dz - \int_{-1}^1 \mu \varphi(z,\cdot) \psi^B(z,\cdot) d\mu \bigg|_0^x + \int_0^x \int_{-1}^1 (\psi^B - (T^B)^4) \varphi d\mu dz = 0,\label{eq:w12}
\end{align}
Due to \eqref{eq:convT}, we can pass to the limit 
\begin{align}
    \int_{0}^x \partial_z T^B \partial_z h dz \to \int_0^x \partial_z T \partial_z h dz ,
    \quad 
    \text{ as } B\to \infty.
\end{align}
Due to \eqref{eq:convT}-\eqref{eq:convpsi2}, we can pass to the limit 
\begin{align}
    \int_0^x \int_{-1}^1 \mu \psi^B \partial_z \varphi d\mu dz \to \int_0^x \int_{-1}^1 \mu \psi \partial_z \varphi d\mu dz,\quad \text{as } B\to \infty,
\end{align}
Using \eqref{eq:convpsi2}, we can pass to the limit as $B\to\infty$,
\begin{align}
    &\int_0^x \int_{-1}^1 (\psi^B - (T^B)^4) h d\mu dz  \to \int_0^x \int_{-1}^1 (\psi - T^4) h d\mu dz,\\
    &  \int_0^x \int_{-1}^1 (\psi^B - T^4) \varphi d\mu dz \to  \int_0^x \int_{-1}^1 (\psi - T^4) \varphi d\mu dz .
\end{align}
By the continuity of the trace operator, we get from the above facts,
\begin{align}
    \partial_x T^B(x) \to \partial_x T(x), \quad \text{almost everywhere.}
\end{align}
Due to $\mu \partial_x \psi + \psi \in C_{\rm loc}(\mathbb{R}_+)$, $\psi^B(x,\mu) \rightharpoonup \psi(x,\mu)$ in $L_{\mu}^2([-1,1])$ almost everywhere, 
\begin{align}
    \int_{-1}^1 \mu \varphi \psi^B d\mu \to \int_{-1}^1 \mu \varphi \psi d\mu.
\end{align}
Combining the above limits, we can pass to the limit $B\to\infty$ in \eqref{eq:w11}-\eqref{eq:w12} and 
\begin{align}
    &\int_{0}^x \partial_z T \partial_z h - \partial_z T(z) h(z)\bigg|_0^x + \int_0^x \int_{-1}^1 (\psi - (T)^4) h d\mu dz = 0,\label{eq:w/1}\\
    &\int_0^x \int_{-1}^1 \mu \psi \partial_z \varphi d\mu dz - \int_{-1}^1 \mu \varphi(z,\cdot) \psi(z,\cdot) d\mu \bigg|_0^x + \int_0^x \int_{-1}^1 (\psi - T^4) \varphi d\mu dz = 0 \label{eq:w/2}
\end{align}
holds.
Therefore, we have proved that there exists a weak solution $(T,\psi) \in C_{\rm loc}(\mathbb{R}_+)\times L^2_{\rm loc}(\mathbb{R}_+\times [-1,1])$ to system \eqref{eq:m1}-\eqref{eq:m2}.
\end{proof}

    
\subsection{Weighted estimate}\label{ProofWEightEst1}
Next we show the weighted estimate \eqref{eq:es-2}.

\begin{proof}[Proof of Theorem \ref{thm:ex}, weighted estimate]

Due to the lower semi-continuous of the norm $\|\cdot\|_{L^2([0,x])}$, the estimate \eqref{eq:estB1} and the strong convergence \eqref{eq:convT} of $T$ and the weak convergence \eqref{eq:h1conv} of $\partial_x T^B$ imply 
\begin{align}\label{eq:ec/1}
    \int_0^x 4T^3 |\partial_z T|^2 dz \le \liminf_{B \to\infty} \int_0^x 4 (T^B)^3 |\partial_z T^B|^2 dz. 
\end{align} 
Due to the weak convergence \eqref{eq:convpsi2} of $\psi^B - (T^B)^4$,
\begin{align}\label{eq:ec/2}
    \int_0^x \int_{-1}^1 (\psi-T^4)^2 d\mu dz \le \liminf_{B\to\infty} \int_0^x \int_{-1}^1 (\psi^B - (T^B)^4)^2 d\mu dz.
\end{align}
From \eqref{eq:estB1},  $\langle \mu (\psi^B - (T^B)^2)^2 \rangle$ is non-negative and uniformly bounded in $L^\infty_{\rm loc}(\mathbb{R}_{+})$. Hence 
\begin{align*}
    \langle \mu (\psi^B - (T^B)^4)^2 \rangle \rightharpoonup^* \langle \mu (\psi - T^4)^2 \rangle,\quad \text{weakly-* in }L^\infty_{\rm loc}(\mathbb{R}_{+}).
\end{align*}
Therefore, 
\begin{align}
    \limsup_{B \to \infty}  \langle \mu (\psi^B - (T^B)^4)^2 \rangle \ge \langle \mu (\psi - T^4)^2 \rangle \ge 0.
\end{align}
Combining the above inequality with \eqref{eq:ec/1}-\eqref{eq:ec/2}, we can take $\limsup$ on \eqref{eq:estB1} and using $\limsup_{n\to\infty} (a_n + b_n) \ge \limsup_{n\to\infty} a_n +\liminf_{n\to\infty}b_n$ to get 
\begin{align}\label{eq:es-1}
    &\int_0^x 4T^3 |\partial_z T|^2 dz + \int_0^x \int_{-1}^1 (\psi-T^4)^2 d\mu dz + \frac12 \langle \mu (\psi(x,\cdot)-T^4(x))^2\rangle \nonumber\\
    &\quad + \frac12 \int_{-1}^0 |\mu| (\psi(0,\cdot)-T_b^4)^2 d\mu \le \frac12 \int_0^1 \mu (\psi_b- T_b^4)^2 d\mu,
\end{align}
holds for any $x\in\mathbb{R}_+$. Note in the above $\langle \mu (\psi(x,\cdot)-T^4(x))^2\rangle \ge 0$ for any $x\in\mathbb{R}_+$.

Similarly, we can pass to the limit $B\to\infty$ in \eqref{eq:estB2} and use the weak convergence up to a subsequence implied by \eqref{eq:estB2}:
\begin{align*}
    & e^{\alpha x} |\partial_x (T^B)^{\frac52}| \rightharpoonup e^{\alpha x} |\partial_x T^{\frac52}|,\quad \text{weakly in }L^2_{\rm loc}(\mathbb{R}_+),\\
    & e^{\alpha x} (\psi^B - (T^B)^4) \rightharpoonup e^{\alpha x} (\psi - T^4),\quad \text{weakly in } L^2_{\rm loc}(\mathbb{R}_+\times [-1,1]), \\
    &e^{2\alpha x} \langle \mu (\psi^B - (T^B)^4)^2 \rangle \rightharpoonup^* e^{2\alpha x} \langle \mu (\psi-T^4)^2 \rangle,\quad \text{weakly-* in }L^\infty_{\rm loc}(\mathbb{R}_+),\\
    & \int_{-1}^0 |\mu|(\psi^B(0,\cdot)-T_b^4)^2 d\mu \to 
    \int_{-1}^0 |\mu| (\psi(0,\cdot)-T_b^4)^2 d\mu,\quad \text{weakly-* in }L^\infty_{\rm loc}(\mathbb{R}_+),
\end{align*}
Consequently, we obtain  
\begin{align*}
    &\int_0^x e^{2\alpha z} 4T^3 |\partial_z T|^2 dz \le \liminf_{B \to\infty} \int_0^x e^{2\alpha} 4 (T^B)^3 |\partial_z T^B|^2 dz , \\
    &\int_0^x \int_{-1}^1 e^{2\alpha z} (\psi-T^4)^2 d\mu dz \le \liminf_{B\to\infty} \int_0^x \int_{-1}^1 e^{2\alpha z} (\psi^B - (T^B)^4)^2 d\mu dz, \\
    &\limsup_{B \to \infty}  \langle \mu (\psi^B - (T^B)^4)^2 \rangle \ge \langle \mu (\psi - T^4)^2 \rangle \ge 0, \\
    &\lim_{B\to\infty} \int_{-1}^0 |\mu|(\psi^B(0,\cdot)-T_b^4)^2 d\mu = \int_{-1}^0 |\mu| (\psi(0,\cdot)-T_b^4)^2 d\mu.
\end{align*}
Therefore, we take $\limsup_{B\to\infty}$ in \eqref{eq:estB2} and use the above inequalities to get 
\begin{align}
    &\int_0^x e^{2\alpha z} 4T^3 |\partial_z T|^2 dx + (1-\alpha) \int_0^x \int_{-1}^1 e^{2\alpha z} (\psi-T^4)^2 d\mu dz \nonumber\\
    &\quad + \frac12  e^{2\alpha x}\langle \mu (\psi(x,\cdot) - T^4(x))^2 \rangle + \frac12 \int_{-1}^0 |\mu| (\psi(0,\cdot)-T_b^4)^2 d\mu \nonumber \\
    &\quad \le \frac12 \int_0^1 \mu (\psi_b- T_b^4)^2 d\mu,\quad \text{for any }x\in\mathbb{R}_+.
\end{align}
From the above estimate, $e^{2\alpha x}\langle \mu (\psi-T^4)^2\rangle$ is bounded, hence $\langle \mu (\psi-T^4)^2 \rangle$ vanishes as $x\to\infty$. Since $\alpha<1$ is arbitrary, we also have $e^{2\alpha x}\langle \mu (\psi-T^4)^2\rangle \to 0$ as $x\to\infty$. We thus let $x\to\infty$ in the above inequality and get the estimate \eqref{eq:es-2}.

\end{proof}

\subsection{Exponential decay}\label{ProofExpo1}
Now, we will  establish the exponential decay property of solutions to system \eqref{eq:m1}-\eqref{eq:m2} using the estimate 
\eqref{eq:es-2}. 
\begin{proof}[Proof of Theorem \ref{thm:ex},  {exponential decay}]
    Let $x_2 > x_1 \ge0$, we have from \eqref{eq:es-2},
    \begin{align*}
        |T^{\frac52}(x_1) - T^{\frac52}(x_2) | &= \left|\int_{x_1}^{x_2}\partial_z T^{\frac52} dz\right| = \left|\int_{x_1}^{x_2}e^{-\alpha x }e^{\alpha z}\partial_z T^{\frac52} dz\right| \\
        &\le \left(\int_{x_1}^{x_2} e^{-2\alpha z} dz \right)^{\frac12} \left(\frac{25}{16}\int_{x_1}^{x_2} e^{2\alpha z} 4T^3|\partial_z T|^2 dz\right)^{\frac12} \\
        &\le \frac{5}{4}\frac{1}{\sqrt{2\alpha}} \left(e^{-2\alpha x_1}-e^{-2\alpha x_2}\right)^{\frac12}  \left(\frac{1}{2} \int_0^1 \mu (\psi_b-T_b^4)^2 d\mu\right)^{\frac12}.
    \end{align*}
    Taking $x_1,x_2 \ge -(2\log \varepsilon) /2\alpha$ with $M>0$ large enough, we can get 
    \begin{align*}
        (e^{-\alpha x_1}-e^{-\alpha x_2}) \le e^{-\alpha x_1}+e^{-\alpha x_2} \le 2\varepsilon^2,
    \end{align*}
    and so
    \begin{align*}
        |T^{\frac52}(x_1)-T^{\frac52}(x_2)| \le \varepsilon\frac{5}{4} \frac{1}{\sqrt{2\alpha}} \left( \int_0^1 \mu (\psi_b-T_b^4)^2 d\mu\right)^{\frac12}
    \end{align*}
    hence 
    \begin{align*}
        \lim_{x\to\infty} T(x) = T_{\infty} \text{ exists and $0\le T_\infty\le \gamma$ is finite.}
    \end{align*}

    Next we show $|T(x)-T_\infty|$ decays to zero exponentially as $x$ goes to infinity. 
    We first show that $\partial_x T (x)= \langle \mu \psi(x,\cdot)\rangle$ holds for any $x\in\mathbb{R}_+$. To see this, we follow the work \cite{klar2001numerical} and compare \eqref{eq:m1} with \eqref{eq:m2} to get
    \begin{align}
        \partial_x (\partial_x T -  \langle \mu \psi \rangle) = 0.
    \end{align}
    Let $j:=\partial_x T -  \langle \mu \psi \rangle$. Multiplying \eqref{eq:m2} by $\mu$ and integrating over $[-1,1]$ gives
    
        \begin{align}\label{KlarEst}
        \partial_x \langle \mu^2 \psi \rangle = \langle \mu \psi \rangle.
    \end{align}
We thus get $j= \partial_x(T+\langle \mu^2 \psi \rangle).$ By the boundness of $T,\psi$, $j=0$ and $T(x)+\langle \mu^2 \psi (x,\cdot)\rangle$ is a constant, i.e.
    \begin{align}\label{eq:pT=}
        \partial_x T(x) = \langle \mu \psi(x,\cdot) \rangle, \quad  \text{ for all}\,\,\, \forall x\in \mathbb{R}_{+}.
    \end{align}
    From the estimate \eqref{eq:es-2}, we obtain 
    \begin{align*}
        (1-\alpha) \int_0^x \int_{-1}^1 e^{2\alpha z} (\psi-T^4)^2 d\mu dz \le \frac12 \int_0^1 \mu (\psi_b-T_b^4)^2 d\mu,
    \end{align*}
    Therefore, 
    \begin{align}\label{eq:mupsi2}
        \int_0^x e^{2\alpha z} |\langle \mu \psi \rangle|^2 dz &= \int_0^x e^{2\alpha z} \left(\int_{-1}^1 \mu (\psi-T^4)d\mu \right)^{2}dz \nonumber \\
        &\le \int_0^x e^{2\alpha z} \left(\int_{-1}^1 \mu^2 d\mu \right) \left(\int_{-1}^1 (\psi-T^4)^2 d\mu \right)dz \nonumber\\
        &\le \frac{2}{3}\int_0^x \int_{-1}^1 e^{2\alpha z} (\psi-T^4)^2 d\mu dz \nonumber \\
        &\le \frac{1}{3(1-\alpha)} \int_0^1 \mu (\psi_b-T_b^4)^2 d\mu.
    \end{align}
    Due to the relation \eqref{eq:pT=}, the above inequality is equivalent to
    \begin{align}\label{eq:pt2}
        \int_0^x e^{2\alpha z} |\partial_z T|^2 dz \le  \frac{1}{3(1-\alpha)} \int_0^1 \mu (\psi_b-T_b^4)^2 d\mu.
    \end{align}
    We get from the above inequality that 
    \begin{align}\label{eq:tdecay1}
        |T(x)-T_\infty| &= \left|-\int_x^\infty \partial_z T dz\right|  = \left|\int_x^\infty e^{-\alpha z} e^{\alpha z} \partial_z T dz \right| \nonumber\\
        &\le \left(\int_x^\infty e^{-2\alpha z} dx\right)^{\frac12} \left(\int_x^\infty e^{2\alpha z} |\partial_z T|^2 dx \right)^{\frac12}\nonumber\\
        & \le \frac{1}{\sqrt{2\alpha}} e^{-\alpha x} \left(\frac{1}{3(1-\alpha)} \int_0^1 \mu (\psi_b-T_b^4)^2 d\mu\right)^{\frac12},
    \end{align}
    and thus $T$ decays exponentially to some constant $T_\infty$ and the first inequality of \eqref{eq:thm.ex} holds. A direct consequence of the above estimate is \eqref{eq:Tbd}.

    Moreover, using \eqref{eq:m1} and the estimate \eqref{eq:es-2}, we have 
    \begin{align*}
        \int_0^x e^{2\alpha z} |\partial_z^2 T|^2 dz &= \int_0^x e^{2\alpha z} |\langle \psi-T^4 \rangle|^2 dz \le \int_0^x \int_{-1}^1 e^{2\alpha z}(\psi-T^4)^2 d\mu dz \\
        &\le \frac12 \int_0^1 \mu (\psi_b - T_b^4)^2 d\mu.
    \end{align*}
    By \eqref{eq:pt2} and the above inequality, the Barbalat's lemma (\cite{farkas2016variations}) implies 
    \[
\partial_x T(x) \to 0, \quad \text{ as } x\to \infty.
    \] 
We can thus get 
    \begin{align} \label{eq:pTdecay}
        |\partial_x T(x)| &= \left|\int_x^\infty \partial_z^2 T dz\right| = \left|\int_x^\infty e^{-\alpha z} e^{\alpha z} \partial_z^2 T dz \right| \nonumber\\
        &\le \left(\int_x^\infty e^{-2\alpha z} dz \right)^{\frac12} \left(\int_x^\infty e^{2\alpha z} |\partial_z^2 T|^2 dz\right)^{\frac12} \nonumber\\
        &\le \frac{1}{\sqrt{2\alpha}} e^{-\alpha z} \left(\frac12 \int_0^1 \mu (\psi_b - T_b^4)^2 d\mu\right)^{\frac12}.
    \end{align}

\smallskip 

Finally, we show the decay of $\psi$. We can pass to the limit $B\to\infty$ in the formula \eqref{eq:sol/mu+}-\eqref{eq:sol/mu-} and show $\psi$ satisfies the formula \eqref{eq:sol1/mu+}-\eqref{eq:sol1/mu-}. Therefore,
    for $\mu>0$,
    \begin{align}\label{eq:dec1}
        &|\psi(x,\mu) - T_\infty^4|\\
         &= \left|\psi_b e^{-\frac{x}{\mu}} + \int_0^x \frac{1}{\mu} e^{-\frac{x-s}{\mu}} T^4(s) ds - T_\infty^4 \right| \nonumber\\
        & = \left|\psi_b e^{-\frac{x}{\mu}} + \int_0^x \frac{1}{\mu} e^{-\frac{x-s}{\mu}}(T^4(s) - T_\infty^4) ds - T^4_\infty e^{-\frac{x}{\mu}} \right| \nonumber\\
        &\le  |\psi_b-T_\infty^4| e^{-\frac{x}{\mu}} + \int_0^x \frac{1}{\mu} e^{-\frac{x-s}{\mu}}|T^4(s) - T_\infty^4| ds  \nonumber \\
        &\le |\psi_b - T_b^4| e^{-\frac{x}{\mu}} + |T_b^4-T_\infty^4| e^{-\frac{x}{\mu}} \nonumber\\
        &\quad + \int_0^x \frac{1}{\mu}e^{-\frac{x-s}{\mu}} |T(s)-T_\infty|\cdot (T^2(s)+T_\infty^2)(T(s)+T_\infty) ds \nonumber\\
        &\le  |\psi_b - T_b^4| e^{-\frac{x}{\mu}} + 4(T_b+M_{\alpha})^3 M_{\alpha} e^{-\frac{x}{\mu}} + \int_0^x \frac{1}{\mu}e^{-\frac{x-s}{\mu}} M_{\alpha} e^{-\alpha x} 4(T_b+2M_{\alpha})^3 ds \nonumber\\
        &= |\psi_b - T_b^4| e^{-\frac{x}{\mu}} + 4(T_b+M_{\alpha})^3 M_{\alpha} e^{-\frac{x}{\mu}} + 4(T_b+2M_{\alpha})^3M_{\alpha} \frac{1}{1-\mu \alpha}(e^{-\alpha x} - e^{-\frac{x}{\mu}}) \nonumber\\
        &\le |\psi_b - T_b^4| e^{-\frac{x}{\mu}} + 4(T_b+2M_{\alpha})^3M_{\alpha} \frac{1}{1- \mu\alpha}e^{-\alpha x},        
    \end{align}
    which gives the second inequality of \eqref{eq:thm.ex},
    and for $\mu<0$,
    \begin{align}\label{eq:dec2}
       &| \psi(x,\mu) - T_\infty^4| \nonumber\\
       &= \left|- \int_x^\infty \frac{1}{\mu} e^{-\frac{x-s}{\mu}}T^4(s) ds - T_\infty^4 \right|\nonumber\\
       & = \left|- \int_x^\infty \frac{1}{\mu} e^{-\frac{x-s}{\mu}}T^4(s) ds  + \int_x^\infty \frac{1}{\mu} 
       e^{-\frac{x-s}{\mu}} T^4_\infty ds\right| \nonumber\\
       &\le -\int_x^\infty \frac{1}{\mu} 
       e^{-\frac{x-s}{\mu}}  |T^4 - T_\infty^4| ds \nonumber \\
       &\le -\int_x^\infty \frac{1}{\mu} e^{-\frac{x-s}{\mu}} |T(s)-T_\infty|\cdot (T^2(s)+T_\infty^2)(T(s)+T_\infty) ds \nonumber\\
       &\le -\int_x^\infty e^{-\frac{x-s}{\mu}} M_{\alpha} e^{-\alpha x} 4(T_b+2M_{\alpha})^3 ds \nonumber\\
       &= -4(T_b+2M_{\alpha})^3 M_{\alpha}\frac{1}{-1+\mu\alpha} e^{-\alpha x} \nonumber\\
       &\le 4(T_b+2M_{\alpha})^3M_{\alpha}\frac{1}{1-\mu\alpha} e^{-\alpha x},
    \end{align}
    which gives the last inequality of \eqref{eq:thm.ex} and finishes the proof of Theorem \ref{thm:ex}. 
    \end{proof}





    \begin{rem}
        For the case $T_b=0,\psi_b=0$, it is obvious that zero is the solution to system \eqref{eq:m1}-\eqref{eq:m2}.
        For the case $T_b >0$, $\psi_b=0$, we assume $T(x)=0$ on $(a,b)\in \mathbb{R}_+$. Then the formula (see \eqref{eq:sol/mu+}-\eqref{eq:sol/mu-})
        \begin{align*}
        \langle \psi(x ,\cdot)\rangle =\int_0^1 \psi_b(\mu) e^{-\frac{x}{\mu}} d\mu + \int_0^1\int_0^\infty \frac{1}{\mu}e^{-\frac{|x-s|}{\mu}}T^4(s)dsd\mu
        \end{align*}
        implies $\langle \psi(x,\cdot) \rangle >0$ for any $ x\in (a,b)$.
        From \eqref{eq:m1}, $\partial_x^2 T (x) = -\langle \psi(x,\cdot) - T^4 (x)\rangle = -\langle \psi (x,\cdot)\rangle <0$ and thus $T(x)$ is concave on the interval $(a,b)$, which contradicts to the assumption $T(x)=0$ on this interval. Therefore, $T(x)>0$ for any $x\in \mathbb{R}_+$.
        For the case $T_b=0$ and $\psi_b(0,\mu) \neq 0$ on the interval $\mu \in (c,d) \subset (0,1).$ Then $\int_0^1 \psi_b(\mu) d\mu >0$ and by the above formula, $\langle \psi(x,\cdot) \rangle >0$. Hence by the same contradiction argument, we have $T(x)>0$ for $x\in\mathbb{R}_+$. Finally for the case $T_b, \psi_b$ both are not zero,  $\int_0^1 \psi_b d\mu >0$ and so $\langle \psi(x,\cdot) \rangle>0$. Therefore, following the same contradiction argument, $T(x) >0$ for all $x\in\mathbb{R}_+$. We conclude that in all cases $T(x)>0$ for all $x\in \mathbb{R}_+$ if $T_b,\psi_b$ are not both zero. Here whether $\lim_{x\to\infty} T(x)$ is positive or not is not known.

    \end{rem}
      \section{Linearized  problem}\label{LMB0}
This section focus on the study of the linearized system \eqref{eq:l-1/i}-\eqref{eq:l-2/i} of \eqref{eq:m1}-\eqref{eq:m2}. We will present the spectral assumption and prove Theorem \ref{thm2}.
Let $T$ be a non-trivial solution to system \eqref{eq:m1}-\eqref{eq:m2}. We consider system \eqref{eq:l-1/i}-\eqref{eq:l-2/i}:
    \begin{align}
        &\partial_x^2 g + \langle \phi - 4T^3 g \rangle=\langle S_1 \rangle,\label{eq:l-1}\\
        &\mu\partial_x \phi + (\phi - 4T^3 g ) =S_1,\label{eq:l-2}
    \end{align}
    with boundary conditions 
    \begin{align}
        &g(0) = 0, \label{eq:l-1b}\\
        &\phi(0,\mu) = \phi_b(\mu), \text{ for any }\mu \in (0,1],\label{eq:l-2b}
    \end{align}
    where the source terms $S_1$ is a function that decays to zero at infinity. Since $S_1$ does not have a definite sign, the solutions to the above system are not necessarily positive. Moreover, $T$ may not be monotonic.
    Therefore, the monotonicity technique used in section \ref{section2} to show the existence of solutions for the nonlinear system \eqref{eq:m1}-\eqref{eq:m2} is not applicable here for the linearized model. Instead, we use the Banach fixed point theorem to show the existence of the linearized model on a bounded interval $[0,B]$.  
    Then, a uniform weighted estimate is derived, allowing us to pass to the limit $B\to\infty$ and show the existence for system \eqref{eq:l-1}-\eqref{eq:l-2}. Finally, a contradiction argument is used for showing the uniqueness of solutions.

To establish the stability estimates and show the existence and uniqueness for system \eqref{eq:l-1}-\eqref{eq:l-2}, a spectral assumption is proposed. With the help of the spectral assumption, we derive the weighted estimate on solutions of system \eqref{eq:l-1}-\eqref{eq:l-2}. An explanation of this assumption is given in section \ref{LMB1} and the assumption will be shown later in the last section to hold when boundary conditions are close to the well-prepared case.    

\subsection{The spectral assumption} \label{LMB1}
We consider the spectrum of the nonlinear operator $\mathcal{F}$ defined in \eqref{eq:opF} with absorbing boundary conditions. First, according to the theory of linear transport operator, there exists $(g,\phi)\neq (0,0)$ such that 
\begin{align}
    \mathcal{F}(g,\phi)  = 0,
\end{align} 
hence $0$ is an eigenvalue (see \cite[Appendix 1]{bardos1984diffusion}). Next we show that all eigenvalues lie in $\Re z \le 0$ under the spectral assumption. 
\begin{lemma}\label{lem:ls}
    Under the spectral assumption \ref{asA}, the spectrum of $\mathcal{F}$ in $L^p(\mathbb{R}_+)\times L^p(\mathbb{R}_+\times[-1,1])$ (with $1\le p \le \infty$), is contained in the half-plane. i.e., $\Re z \le 0$  where $0$ belongs to the spectrum.
\end{lemma}
Before  giving the proof of  this lemma, we will show some inequalities related to the spectral assumption \ref{asA}. 

\begin{lemma}\label{lem.0}
    Assume  $T\in C^1_{\rm loc}(\mathbb{R}_+)$  satisfies the spectral assumption \ref{asA}. Then the following  inequality
    \begin{align}\label{eq:spM}
        M\int_0^\infty (2T^{\frac32})^2 |\partial_x g|^2 dx  \ge  4\int_0^\infty |\partial_x (2T^{\frac32})|^2 g^2 dx,
    \end{align}
holds for the same constant $M<1$ given in \ref{asA} and any function $g \in C_{\rm loc}^{1}(\mathbb{R}_{+})$ with $g(0)=0$.
\end{lemma}
\begin{proof}
    Let $F\in C_{\rm loc}(\mathbb{R}_+)$ be any function defined on the half-line and $F(x)\ge 0$ for any $x\in\mathbb{R}_+$. Taking $f=\int_0^x F(t) dt$ in \eqref{eq:spassump2} leads to 
    \begin{align}
        M\int_0^\infty e^{2\beta x} (2T^{\frac32})^2 F^2(x) dx  \ge  4\int_0^\infty e^{2\beta x} |\partial_x (2T^{\frac32})|^2 \left(\int_0^x F(t) dt\right)^2 dx,
    \end{align}
    for  the same constant $M<1$ as in the spectral assumption \ref{asA}. Let $G(x) = e^{\beta x}F(x) \ge 0$, the above inequality implies
    \begin{align}
        M \int_0^\infty  (2T^{\frac32})^2 G^2(x)dx &\ge 4 \int_0^\infty  e^{2\beta x}|\partial_x(2T^{\frac32})|^2 \left(\int_0^x e^{-\beta t}G(t) dt \right)^2 dx \nonumber\\
        &= 4 \int_0^\infty  |\partial_x(2T^{\frac32})|^2 \left(\int_0^x e^{\beta (x-t)}G(t) dt \right)^2 dx \nonumber\\
        &\ge 4 \int_0^\infty  |\partial_x(2T^{\frac32})|^2 \left(\int_0^x G(t) dt \right)^2 dx,
    \end{align}
    since $e^{\beta(x-t)} \ge 1$ for any $t\le x$. Due to $\int_0^x |G(t)|dt \ge \int_0^x G(t)dt$, the above inequality also holds for all $G(x) \in C_{\rm loc}(\mathbb{R}_+)$. Let $g(x) = \int_0^x G(t) dt$, then \eqref{eq:spM} holds for any $g\in C^1_{\rm loc}(\mathbb{R}_+)$ satisfying $g(0)=0$.
\end{proof}

Finally we give some inequalities implied by the spectral assumption that are used in the proof of existence for system \ref{eq:l-1}-\eqref{eq:l-2}.
    \begin{lemma}\label{LemmaSA}
        Assume the spectral assumption \ref{asA} is fulfilled, then 
\begin{align}\label{sp1}
            \int_0^\infty 4T^3|\partial_x f|^2 dx + \int_0^\infty \partial_x (4T^3) f\partial_x f dx \ge  0
        \end{align}
        and
        \begin{align} \label{sp2}
            \int_0^\infty e^{2\beta x} 4T^3|\partial_x f|^2 dx + \int_0^\infty e^{2\beta x} \partial_x (4T^3) f\partial_x f dx \ge 0
        \end{align}
       hold for any function $f\in C^1$ such that  $f(0)=0$ and for $\beta \in [0,\beta_0]$ where $\beta_{0}$ is given in Lemma \ref{lem.spcond}.
    \end{lemma}
    \begin{proof}
    The spectral assumption \ref{asA} implies  
    \begin{align}
        \int_0^\infty 4T^3 |\partial_x f|^2 dx \ge \int_0^\infty 36 T|\partial_x T|^2 f^2 dx.
    \end{align}
    By Young's inequality, we deduce that 
        \begin{align*}
            \left|\int_0^\infty \partial_x (4T^3) f \partial_x f dx\right| &\le \frac12 \int_0^\infty \frac{1}{4T^3} |\partial_x(4T^3)|^2 f^2 dx + \frac12 \int_0^\infty 4T^3 |\partial_x f|^2 dx \nonumber\\
            & = \frac12 \int_0^\infty 36 T |\partial_x T|^2 f^2 dx + \frac12 \int_0^\infty 4T^3 |\partial_x f|^2 dx.
        \end{align*}
        Combining this with the previous inequality, we obtain
        \begin{align*}
            &\int_0^\infty 4T^3|\partial_x f|^2 dx + \int_0^\infty \partial_x (4T^3) f\partial_x f dx \nonumber\\
            &\quad \ge  \int_0^\infty 4T^3|\partial_x f|^2 dx - \frac12 \int_0^\infty 36 T |\partial_x T|^2 f^2 dx - \frac12 \int_0^\infty 4T^3 |\partial_x f|^2 dx \nonumber\\
            &\quad = \frac12  \int_0^\infty 4T^3|\partial_x f|^2 dx -  \frac12 \int_0^\infty 36 T |\partial_x T|^2 f^2 dx \ge 0,
        \end{align*}
        and thus \eqref{sp1} holds.
        
        Similarly, the spectral assumption \ref{asA} also implies for $\beta \in [0,\beta_0]$,
        \begin{align}
            \int_0^\infty e^{2\beta x} 4T^3 |\partial_x f|^2 dx \ge \int_0^\infty e^{2\beta x} 36 T|\partial_x T|^2 f^2 dx.
        \end{align}
        This implies 
        \begin{align}
            &\int_0^\infty e^{2\beta x} 4T^3|\partial_x f|^2 dx + \int_0^\infty e^{2\beta x} \partial_x (4T^3) f\partial_x f dx\nonumber\\
            &\quad \ge \int_0^\infty e^{2\beta x} 4T^3|\partial_x f|^2 dx - \frac12 \int_0^\infty e^{2\beta x} 4T^3|\partial_x f|^2 dx - \frac12 \int_0^\infty e^{2\beta x} 36 T|\partial_x T|^2 f^2 dx \nonumber\\
            &\quad = \frac12  \int_0^\infty e^{2\beta x} 4T^3|\partial_x f|^2 dx -  \frac12 \int_0^\infty e^{2\beta x} 36 T|\partial_x T|^2 f^2 dx \ge 0,
        \end{align}
        which is \eqref{sp2} and the proof is finished.
       \end{proof}

We next give the proof of Lemma \ref{lem:ls}.
       \begin{proof}[Proof of Lemma \ref{lem:ls}]
Suppose $\lambda$ is an eigenvalue of $\mathcal{F}$ associated to the  eigenvector $(g,\phi)$. Then 
        \begin{align}
            \lambda \left(\begin{array}[]{cc}
                g \\\phi
            \end{array}\right) = \mathcal{F}(g,\phi) = \left(\begin{array}[]{cc}
                \partial_x^2 g + \langle \phi - 4T^3 g \rangle \\
               - \mu \partial_x \phi -( \phi-4T^3 g)
            \end{array}\right).
        \end{align}
        We multiply the above equation by $(4T^3 g,\phi)$ and integrate over $\mathbb{R}_+\times [-1,1]$ to get 
        \begin{align}\label{eq:ls/1}
            &\lambda \int_{0}^\infty 4T^3 g^2 dx -+\lambda \int_0^\infty \phi^2 dx \nonumber\\
            &\quad = \int_0^\infty 4T^3 g\partial_x^2 g dx - \int_0^\infty \int_{-1} \mu \partial_x \frac{\phi^2}{2} d\mu dx - \int_0^\infty \int_{-1}^1 (\phi-4T^3 g)^2 d\mu dx.
        \end{align}
        The spectral assumption implies 
        \begin{align}
            \int_0^\infty 4T^3 g\partial_x^2 g dx =- \int_0^\infty 4T^3 |\partial_x g|^2 dx - \int_0^\infty \partial_x(4T^3) g \partial_x g dx \le 0.
        \end{align}
        The absorbing boundary condition implies 
        \begin{align}
            \int_0^\infty \int_{-1} \mu \partial_x \frac{\phi^2}{2} d\mu dx = 0.
        \end{align}
        Taking the above two inequalities into \eqref{eq:ls/1}, we obtain 
        \begin{align}
            &\lambda \int_{0}^\infty 4T^3 g^2 dx + \lambda \int_0^\infty \phi^2 dx \le \int_0^\infty \int_{-1}^1 (\phi-4T^3 g)^2 d\mu dx \le 0.
        \end{align}
        Hence if $(g,\phi)\neq (0,0)$, then $\lambda\le 0$ and finishes the proof of Lemma \ref{lem:ls}.
    \end{proof}
 Assuming $S_1$, $S_2$ are two given functions decaying exponentially to zero, by Lemma \ref{lem:ls}, there exists a unique solution to the following problem 
    \begin{align}
        &\eps(g_\eps^1,\phi_\eps^1)^T + \mathcal{F}(g_\eps^1,\phi_\eps^1)=(-S_1,S_2)^T, \\
        &g_\eps^1(0) =0 ,\quad \phi_\eps^1(0,\mu)=0,\quad \text{for any }\mu>0.
    \end{align}
    By a superposition argument, there also exists a unique solution to 
    \begin{align}
        &\eps(g_\eps^2,\phi_\eps^2)^T + \mathcal{F}(g_\eps^2,\phi_\eps^2)=0, \\
        &g_\eps^2(0) = g_b,\quad \phi_\eps^2(0,\mu)=\phi_b,\quad \text{for any }\mu>0.
    \end{align}
    Then $g_\eps = g_{\eps}^1 + g_\eps^2$, $\phi_\eps:=\phi_\eps^1 + \phi_\eps^2$ solves 
    \begin{align}
     &   -\eps g_\eps + \partial_x^2 g_\eps + \langle \phi_\eps - 4T^3 g_\eps \rangle = S_1,\\
      &  \eps \phi_\eps + \mu\partial_x\phi_\eps + \phi_\eps - 4T^3 g_\eps = S_2,\\
    &g_\eps(0) =0 ,\quad \phi_\eps(0,\mu)=0,\quad \text{for any }\mu>0.
    \end{align}
    Formally, the above system converges to \eqref{eq:m1}-\eqref{eq:bd2}. However, we need to derive a uniform estimate on the solutions to the above problem. This can done for the Milne problem of linear transport equation using a Maximum principle (\cite{bardos1984diffusion}). However, due to $T$ not being a constant function, we cannot use the Maximum principle to show the uniform boundness of $(g_\eps,\phi_\eps)$. Here we overcome this problem by first showing existence of \eqref{eq:m1}-\eqref{eq:bd2} in a bounded interval and extend the solutions to half-space using uniform estimates.

    \subsection{Existence on the bounded interval}\label{LMB2}
    We now consider system \eqref{eq:l-1}-\eqref{eq:l-2} on the bounded interval $[0,B]$:
    \begin{align}
        \partial_x^2 g^B + \langle \phi^B - 4T^3 g^B \rangle = \langle S_1 \rangle, \label{eq:BBl1}\\
        \mu \partial_x \phi^B + (\phi^B - 4T^3 g^B) = S_1,\label{eq:BBl2}
    \end{align}
    with boundary conditions 
    \begin{align}
        &g^B(0) = 0,\quad \partial_x g^B(B)=0,\label{eq:B/l1}\\
        &\phi^B(0,\mu) = \phi_b(\mu),\quad \phi^B(B,\mu) = \phi^B(B,-\mu),\text{ for any }\mu>0,\label{eq:B/l2}
    \end{align}
    where $T^B$ is the solution to the system \eqref{eq:B1}-\eqref{eq:B2}. 

    Define the weighted space   
    \begin{align*}
        L^2_{m(x)}([0,B]):=\left\{f: \int_0^B m(x) f^2 dx < \infty \right\}
    \end{align*}
    where $m(x)>0$ is a given function.
    We prove the following existence lemma.
    \begin{lemma}\label{lm:exbd}
        Let $S_1=S_1(x)$ be a continuous function on $[0,B]$. Then there exists a unique solution $(g^B,\phi^B) \in C^1([0,B])\times C([0,B]\times [-1,1])$ to  system \eqref{eq:BBl1}-\eqref{eq:BBl2} with boundary conditions \eqref{eq:B/l1}-\eqref{eq:B/l2}, and the solution satisfies 
    \begin{align}\label{eq:est-wB}
        &\int_0^B \int_{-1}^1 e^{2\beta x}(\phi^B - 4(T^B)^3 g^B)^2 d\mu dx +\frac{1}{1-\beta} \int_{-1}^0 |\mu| (\phi^B)^2(0,\cdot) d\mu \nonumber\\
        &\qquad \le  \frac{1}{1-\beta} \int_0^1 \mu \phi_b^2 d\mu + \frac{2}{(1-\beta)^2} \int_0^B e^{2\beta x} |S_1|^2 dx.
    \end{align}
        \end{lemma}
        
    \begin{proof}
        The proof is divided into two steps. In the first step, we prove existence. We then derive the uniform estimate in the second step.
        
        \emph{Step 1:  Existence.}
        For a given $g^B$, the solution to \eqref{eq:BBl2}, according to \eqref{eq:sol/mu+}-\eqref{eq:sol/mu-}, can be expressed as 
    \begin{align}\label{eq:phiexact}
        \phi^B = \left\{\begin{array}{lc}
             \phi_b(\mu) e^{-\frac{x}{\mu}} + \int_0^x \frac{1}{\mu} e^{-\frac{x-s}{\mu}} 4(T^B)^3(s) g^B(s) ds , &\text{for }\mu>0,\\
             \phi_b(-\mu) e^{\frac{2B-x}{\mu}} - \int_0^B \frac{1}{\mu} e^{\frac{2B-x-s}{\mu}} 4(T^B)^3(s) g^B(s)  dx \\
             \qquad- \int_x^B\frac{1}{\mu} e^{-\frac{x-s}{\mu}} 4(T^B)^3(s) g^B(s)ds, &\text{for }\mu<0.
        \end{array}\right.
    \end{align}
    We denote the above relation by the mapping $\phi^B ={\Phi}_{\phi_b}(g^B).$ Then we have 
    \begin{align}
        \langle \phi^B(x,\cdot) \rangle 
        & = \int_0^1 \phi^B(x,\mu) d\mu + \int_{-1}^0 \phi^B(x,\mu) d\mu \nonumber\\
        & =\int_0^1 \phi_b(\mu) e^{-\frac{x}{\mu}} d\mu + \int_0^1 \phi_b(\mu) e^{-\frac{2B-x}{\mu}} d\mu \nonumber\\
        &\quad+ \int_0^B \int_0^1 \frac{1}{\mu} e^{-\frac{|x-s|}{\mu}} 4 (T^B)^3(s)g^B(s) d\mu  ds \nonumber\\&
        \quad +\int_0^B \int_{0}^1 \frac{1}{\mu} e^{-\frac{2B-x-s}{\mu}} 4 (T^B)^3(s)g^B(s) d\mu ds .\label{eq:<phiB>}
    \end{align}
    System \eqref{eq:BBl1}-\eqref{eq:BBl2} is equivalent to the following equation
    \begin{align}\label{eq:gmathcalF}
        -\partial_x^2 g^B + 8T^3 g^B - \langle {\Phi}_{\phi_b}(g^B) \rangle = 2S_1,
    \end{align}
    with 
    \begin{align*}
        g^B (0)=0,\quad \partial_x g^B (B)=0.
    \end{align*}
 To prove the existence for the above equation, we construct a sequence $\{g_k\}_{k=0}^\infty$ with $g^0(x)=g_b$ and $g^k$ is solved via 
    \begin{align}\label{eq:itereq}
        -\partial_x^2 g_k + 8T^3 g_k - \langle {\Phi}_{\phi_b}(g_{k-1})\rangle = 2S_1,
    \end{align}
    with 
    \begin{align*}
        g_k(0)=0, \quad \partial_x g_k(B)=0.
    \end{align*}
    Denote the above mapping by $\mathcal{G}$ with $g_k = \mathcal{G} g_{k-1} $. Then according to \cite[Chpater 2.2.3]{kato2013perturbation}, $\mathcal{G}$ is invertible on the space $C([0,B])$ and thus the above equation has a solution in $g_k \in C([0,B])$. Moreover, by \eqref{eq:itereq}, $g_k \in C^2([0,B])$. Next we show $\mathcal{G}$ is a contraction mapping in the weighted space $L^2_{(16T^6)}([0,B])$. 
    
    We first show $\mathcal{G}$ maps the space $L^2_{(16(T^B)^6)}([0,B])$ onto itself. We multiply \eqref{eq:itereq} by $4(T^B)^3 g_k$ and integrate over $[0,B]$ to get 
    \begin{align}\label{eq:est0B}
       - \int_0^B 4 T^3 g_k \partial_x^2 g_k dx + \int_0^B 32 T^6 g_k^2 dx  = \int_0^B 4 T^3 g_k \langle \Phi_{\phi_b}(g_{k-1})\rangle  dx.
    \end{align}
    Since $T^B$ satisfies the spectral assumption \ref{asA}, by Lemma \ref{LemmaSA}, inequality \eqref{sp1} holds, which implies
    \begin{align}\label{eq:A>0}
       - \int_0^B 4 T^3 g_k \partial_x^2 g_k dx = \int_0^B 4T^3 |\partial_x g_k |^2 dx +  \int_0^B \partial_x(4 T^3) g_k \partial_x g_k dx \ge 0.
    \end{align}
    Then, the Young's inequality implies 
    \begin{align}\label{eq:Fin}
        \int_0^B 4 T^3 g_k \langle \Phi_{\phi_b}(g_{k-1})\rangle dx \le \int_0^B 16 T^6 g_k^2 dx + \frac{1}{4} \int_0^B  |\langle\Phi_{\phi_b}(g_{k-1})\rangle|^2 dx.
    \end{align}
    Using the formula \eqref{eq:<phiB>} and the Young's convolution inequality  we obtain
    \begin{align*}
        \left\|\int_0^x f(s) \int_0^1 \frac{1}{\mu} e^{-\frac{x-s}{\mu}} d\mu  ds \right\|_{L^2([0,B])}^2 &\le \left\| \int_0^1 \frac{1}{\mu} e^{-\frac{x}{\mu}} d\mu \right\|_{L^1([0,B])}^2 \|f\|_{L^2([0,B])}^2 \nonumber\\
        &\le C(B) \|f\|_{L^2([0,B])}^2, 
    \end{align*}
    and 
    \begin{align*}
        \left\|\int_x^B f(s) \int_0^1 \frac{1}{\mu} e^{\frac{x-s}{\mu}} d\mu  ds \right\|_{L^2([0,B])}^2 &\le  \left\| \int_0^1 \frac{1}{\mu} e^{\frac{x}{\mu}} d\mu \right\|_{L^1([0,B])}^2 \|f\|_{L^2([0,B])} \nonumber\\
        &\le C(B) \|f\|_{L^2([0,B])}, 
    \end{align*}
    where $C(B)>0$ is a constant depending only on $B$.
    Similarly,
    \begin{align}
        \left\|\int_0^B \int_0^1 \frac{1}{\mu} e^{-\frac{2B-x-s}{\mu}} f(s,\mu) d\mu \right\|_{L^2([0,B])} \le C(B) \|f\|_{L^2([0,B]\times [-1,1])}
    \end{align}
    By the above inequalities and \eqref{eq:<phiB>}, we deduce 
    \begin{align}
        &\|\langle |\Phi_{\phi_b}(g_{k-1})\rangle\|_{L^2([0,B])}^2   \le  C(B,\phi_b) + C(B) (\|4T^3 g_{k-1} \|_{L^2([0,B])}^2. \nonumber
    \end{align}
    Taking this into \eqref{eq:Fin} leads to 
    \begin{align}
        &\int_0^B 4 T^3 g_k \langle \Phi_{\phi_b}(g_{k-1})\rangle dx \le \int_0^B 16 (T^B)^6 g_k^2 dx  + C(B) \int_0^B 16 T^6 g_{k-1}^2 dx.
    \end{align}
    Combining the above inequality with \eqref{eq:A>0}, \eqref{eq:est0B} implies 
    \begin{align}
        \int_0^B 16 T^6 g_k^2 dx \le   C(B) \int_0^B 16 T^6 g_{k-1}^2 dx.
    \end{align}
    Hence $\mathcal{G}$ maps from $L^2_{16T^6}([0,B])$ onto itself.

We next show that  $\mathcal{G}$ is a contraction mapping.
    Let $h_k = g_k-g_{k-1}$, we have 
    \begin{align}\label{eq:hk}
        \partial_x^2 h_k = 8T^3 h_k - \langle \Phi_0(h_{k-1}) \rangle,
    \end{align}
    with 
    \begin{align*}
        h_k = 0, \quad \partial_x h_k(B) = 0.
    \end{align*}
    We multiply \eqref{eq:hk} by $4(T^B)^3 h_k$ and integrate over $[0,B]$ to get 
    \begin{align}\label{eq:h/1}
        -\int_0^B 4T^3 h_k \partial_x^2 h_kdz  + \int_0^B 32T^6 h_k^2 dx - \int_0^B \langle \Phi_0(h_{k-1}) \rangle 4T^3 h_k dx = 0.
    \end{align}
    The inequality \eqref{sp1}  implied by the spectral assumption on $T^B$ leads to
    \begin{align}\label{eq:h/2}
        -\int_0^B 4T^3 h_k \partial_x^2 h_k dx =  \int_0^B 4T^3 |\partial_x h_k|^2dz + \int_0^B \partial_x(4T^3) h_k \partial_x h_kdz \ge 0.
    \end{align}
    Recalling \eqref{eq:<phiB>}, we have
    \begin{align}\label{eq:h/3}
        \int_0^B 4&T^3 h_k \langle \Phi_{\phi_b=0}(h_{k-1}) \rangle dx\\
        =~&\int_0^B \int_0^B \int_0^1 \frac{1}{\mu} e^{-\frac{|x-s|}{\mu}} 16 T^3(s)h_{k-1}(s) T^3(x) h_k(x) d\mu  ds  dx \nonumber\\
        &+\int_0^B \int_{0}^1 \frac{1}{\mu} e^{-\frac{2B-x-s}{\mu}} 16 T^3(s)h_{k-1}(s) T^3(x) h_k(x)d\mu ds dx\nonumber\\
        =~&- \frac{1}{2}\int_0^B \int_0^B \int_0^1 \frac{1}{\mu} e^{-\frac{|x-s|}{\mu}} (4T^3(s)h_{k-1}(s)-4T^3(x)h_k(x))^2 d\mu  ds  dx \nonumber \\
        & - \frac{1}{2} \int_0^B \int_{0}^1 \frac{1}{\mu} e^{-\frac{2B-x-s}{\mu}} (4T^3(s)h_{k-1}(s)-4T^3(x)h_k(x))^2 d\mu ds dx\nonumber\\
        &+\frac{1}{2}\int_0^B\int_0^B\int_0^1 e^{-\frac{|x-s|}{\mu}} (16T^6(s)h_{k-1}^2(s) + 16 T^3(x)h_k^2(x))d\mu ds dx \nonumber\\
        &+\frac{1}{2}\int_0^B\int_0^B\int_0^1 \frac{1}{\mu} e^{-\frac{2B-x-s}{\mu}} (16T^6(s)h_{k-1}^2(s) + 16 T^3(x)h_k^2(x)) d\mu ds dx \nonumber\\
        &=-\frac{1}{2}\int_0^B \int_0^B \int_0^1 \frac{1}{\mu} e^{-\frac{|x-s|}{\mu}} (4T^3(s)h_{k-1}(s)-4T^3(x)h_k(x))^2 d\mu  ds  dx \nonumber \\
        & - \frac{1}{2} \int_0^B \int_{0}^1 \frac{1}{\mu} e^{-\frac{2B-x-s}{\mu}} (4T^3(s)h_{k-1}(s)-4T^3(x)h_k(x))^2 d\mu ds dx\nonumber\\
        & + \frac{1}{2} \int_0^B \int_0^1 (1-e^{-\frac{x}{s}} + 1-e^{-\frac{B-x}{\mu}} + e^{-\frac{B-x}{\mu}}-e^{-\frac{2B-x}{\mu}}) \nonumber\\
        &\cdot 16 T^6(x)(h_k^2(x)+h_{k-1}^2(x)) d\mu dx \nonumber\\
        \le~& \frac12 \int_0^B \int_0^1 16(2-e^{-\frac{x}{\mu}}-e^{-\frac{2B-x}{\mu}}) T^6(x) (h_k^2(x)+h_{k-1}^2(x)) d\mu dx .\nonumber\\
    \end{align}
    Taking \eqref{eq:h/2} and \eqref{eq:h/3} into \eqref{eq:h/1} gives 
    \begin{align*}
        &\frac12 \int_0^B \int_0^1 16(2+e^{-\frac{x}{\mu}}+e^{-\frac{2B-x}{\mu}})T^6(x) h_k^2(x) d\mu dx \nonumber\\
        &\quad - \frac12 \int_0^B\int_0^1 16(2-e^{-\frac{x}{\mu}} - e^{-\frac{2B-x}{\mu}}) T^6(x) h_{k-1}^2(x) d\mu dx \le 0.
    \end{align*}
    Due to 
    \begin{align*}
        \min_{x\in [0,B]}(e^{-\frac{x}{\mu}} + e^{-\frac{2B-x}{\mu}} )= e^{-\frac{B}{\mu}} + e^{-\frac{B}{\mu}} = 2e^{-\frac{B}{\mu}},
    \end{align*}
    The previous inequality implies 
    \begin{align}\label{eq:<delta}
        \int_0^B 16 T^6 (x)h_k^2(x)dz \le  \delta \int_0^B 16 T^6 (x)h_{k-1}^2(x)dz ,
    \end{align}
    with 
    \begin{align*}
        \delta = \frac{\int_0^1(1-e^{-\frac{B}{\mu}})d\mu}{\int_0^1 (1+e^{-\frac{B}{\mu}}) d\mu} <1.
    \end{align*}
    Then \eqref{eq:<delta} can be written as 
    \begin{align*}
        \|h_k\|^2_{L^2_{16T^6}([0,B])} \le \delta  \|h_{k-1}\|^2_{L^2_{16T^6}([0,B])} ,
    \end{align*}
    which is  
    \begin{align*}
        \|g_k-g_{k-1}\|_{L^2_{16T^6}([0,B])}^2 \le \delta  \|g_{k-1}-g_{k-2}\|_{L^2_{16T^6}([0,B])}^2 ,
    \end{align*}
    with $\delta<1$. Thus the map $\mathcal{G}:L^2_{16T^6}([0,B]) \mapsto L^2_{16T^6}([0,B])$ is a contraction mapping. By the Banach fixed point theorem, there exists a unique solution to the equation \eqref{eq:gmathcalF} in $L^2_{16 T^6}([0,B])$, i.e. there exists a unique solution $g^B \in L^2_{16 T^6}([0,B]) $ to the equation \eqref{eq:gmathcalF}. By equation \eqref{eq:gmathcalF}, $\partial_x^2 g^B \in L^2([0,B])$ and thus $g^B \in C^1([0,B])$.    
    Since $g^B \in L^2_{16 T^6}([0,B]) $ is equivalent to $4(T^B)^3 g^B \in {L^2([0,B])}$. By the formula \eqref{eq:phiexact}, $\phi^B \in C([0,B]\times [-1,1])$.
    
    
    \smallskip


        \emph{Step 2. The weighted estimates.}
        We first  multiply \eqref{eq:BBl1} by $4T^3 g^B$ and integrate over $[0,B]$, we also multiply \eqref{eq:BBl2} by $\phi^B$ and integrate over $[0,B]$ and $[-1,1]$ to get 
            \begin{align}\label{eq:ccc0}
                &- \int_0^B 4T^3 g^B \partial_x^2 g^B dx +  \int_0^B \int_{-1}^1 \mu \partial_x \frac{(\phi^B)^2}{2} d\mu dx \nonumber\\
                &\quad + \int_0^B \int_{-1}^1 (\phi^B - 4T^3 g^B)^2 d\mu dx = \int_0^B \int_{-1}^1 (\phi^B - 4T^3 g^B) S_1 d\mu dx.
            \end{align}
        By the boundary condition \eqref{eq:B/l1} and the inequality \eqref{sp1} implied by the spectral assumption \eqref{sp2},
         we can get
            \begin{align}\label{eq:ccc2}
                -\int_0^B 4 T^3 g^B \partial_x^2 g^B dx 
                =~& \int_0^B 4T^3 |\partial_x g^B|^2 dx + \int_0^B \partial_x (4T^3) g^B\partial_x  (g^B)  
                \ge 0.
            \end{align}
            Using integration by parts and the boundary conditions \eqref{eq:B/l2}, we have 
            \begin{align}\label{eq:ccc/1}
                \int_0^B \int_{-1}^1  \mu \partial_x \frac{(\phi^B)^2}{2} d\mu dx &= \frac12 \int_{-1}^1 \mu (\phi^B)^2 d\mu \bigg|_0^B =  -\frac12 \int_{-1}^1 \mu (\phi^B)^2 d\mu \nonumber\\
                &= -\frac12 \int_{-1}^0 \mu (\phi^B)^2(0,\cdot) d\mu - \frac12 \int_0^1 \mu \phi_b^2 d\mu.
            \end{align}
            Using Young's inequality on the last term of \eqref{eq:ccc0} gives 
            \begin{align}
                &\int_0^B \int_{-1}^1 (\phi^B - 4T^3 g^B) S_1 d\mu dx \nonumber\\
                &\quad\le \xi \int_0^B \int_{-1}^1 (\phi^B-4T^3 g^B)^2 d\mu dx + \frac{1}{4\xi} \int_0^B |S_1|^2 dx \int_{-1}^1 d\mu \nonumber\\
                &\quad = \xi \int_0^B \int_{-1}^1 (\phi^B-4T^3 g^B)^2 d\mu dx + \frac{1}{2\xi} \int_0^B |S_1|^2 dx,
            \end{align}
            where $\xi>0$ is a positive constant and will be chosen later.
            Taking the above inequality and \eqref{eq:ccc2}, \eqref{eq:ccc/1} into \eqref{eq:ccc0} leads to the estimate 
            \begin{align}
                &(1-\xi) \int_0^B \int_{-1}^1 (\phi^B - 4T^3 g^B)^2 d\mu dx 
                +  \frac12 \int_{-1}^0 |\mu| (\phi^B)^2(0,\cdot) d\mu \nonumber\\
                 &\quad \le  \frac12 \int_0^1 \mu \phi_b^2 d\mu + \frac{1}{2\xi} \int_0^B |S_1|^2 dx.
            \end{align} 
            Taking $\xi=\frac12$, the above inequality becomes the weighted estimate \eqref{eq:est-wB} with $\beta=0$.
        
            We next derive the weighted estimate on system \eqref{eq:BBl1}-\eqref{eq:BBl2} for $\beta>0$. Multiplying \eqref{eq:BBl1} by $e^{2\beta x} 4 T^3 g^B$ and \eqref{eq:BBl2} by $e^{2\beta x} \phi^B$ and integrating over $[0,B]$ and $[-1,1]$, we obtain 
            \begin{align}\label{eq:cc/1}
                &-\int_0^B e^{2\beta x} 4T^3 g^B \partial_x^2 g^B dx + \int_0^B \int_{-1}^1 \mu e^{2\beta x} \partial_x \frac{(\phi^B)^2}{2} d\mu dx \nonumber\\
                &\quad + \int_0^B \int_{-1}^1 e^{2\beta x} (\phi^B - 4T^3 g^B)^2 d\mu dx = \int_0^B \int_{-1}^1 e^{2\beta x} (\phi^B - 4T^3 g^B) S_1 dx .
            \end{align}
            Using integration by parts, the inequality \eqref{sp2} implied by the spectral assumption \ref{asA}, and the boundary condition \eqref{eq:B/l1}, we get 
            \begin{align}\label{eq:cc/2}
                &-\int_0^B e^{2\beta x} 4T^3 g^B \partial_x^2 g^B dx \nonumber\\
                &\quad= \int_0^B e^{2\beta x} 4T^3|\partial_x g^B|^2 dx + \int_0^x e^{2\beta x} \partial_x (2T^3) \partial_x (g^B)^2 dx \nonumber\\
                &\qquad + 2\beta \int_0^x e^{2\beta x} 4 T^3 g^B \partial_x g^B dx - e^{2\beta x} 4T^3 g^B \partial_x g^B \bigg|_0^B \nonumber\\
                &\qquad \ge  2\beta \int_0^x e^{2\beta x} 4 T^3 g^B \partial_x g^B dx. 
            \end{align}
        Moreover, using integration by parts, we have
            \begin{align}\label{eq:cc/3}
                &\int_0^B \int_{-1}^1 \mu e^{2\beta x} \partial_x \frac{(\phi^B)^2}{2} d\mu dx \nonumber\\
                &\quad =-\beta \int_0^B \int_{-1}^1  e^{2\beta x} \mu (\phi^B)^2 d\mu dx + \frac12 \int_{-1}^1  e^{2\beta x} \mu (\phi^B)^2 d\mu\bigg|_0^B \nonumber\\
                &\quad = -\beta \int_0^B \int_{-1}^1  e^{2\beta x} \mu (\phi^B)^2 d\mu dx - \frac12 \int_{-1}^1 \mu (\phi^B)^2(0,\cdot)d\mu.
            \end{align}
                Comparing \eqref{eq:BBl1} with \eqref{eq:BBl2} and using the assumption $\langle S_2 \rangle = S_1$, we obtain
            \begin{align}\label{eq:ccc/2}
                \partial_x (\partial_x g^B - \langle \mu \phi^B \rangle) = 0.
            \end{align}
            Due to the boundary conditions \eqref{eq:B/l1}-\eqref{eq:B/l2}, 
            \[\partial_x g^B(B)=0,\quad \langle \mu \phi^B(B,\cdot)\rangle =0.\]
        Therefore, by  \eqref{eq:ccc/2}, we deduce that  
            \begin{align}\label{eq:ccc/-}
                \partial_x g^B (x)= \langle \mu \phi^B(x,\cdot) \rangle,\quad \text{ for any } x\in [0,B].
            \end{align}
            Using the above relation, we have  
            \begin{align}\label{eq:ccc/3}
                \langle \mu (\phi^B)^2 \rangle - 4T^3 g^B =  \langle \mu (\phi^B)^2 \rangle - 4T^3 \langle \mu \phi^B\rangle = \langle \mu (\phi^B-4T^3 g^B)^2\rangle.
            \end{align}
            Consequently, it follows that  
            \begin{align}\label{eq:cc/4}
                &2\beta \int_0^B e^{2\beta x} 4 T^3 g^B \partial_x g^B dx - \beta\int_0^B\int_{-1}^1 e^{2\beta x} \mu (\phi^B)^2 d\mu dx \nonumber\\
                &\quad = -\beta \int_0^B\int_{-1}^1 \mu (\phi^B-4T^3 g^B)^2 d\mu dx.
            \end{align}
            For the right terms of \eqref{eq:cc/1}, we can use Young's inequality to get 
            \begin{align*}
                &\int_0^B \int_{-1}^1 e^{2\beta x}(\phi^B - 4T^3 g^B) S_1 d\mu dx \nonumber\\
                &\quad \le \xi \int_0^B \int_{-1}^1e^{2\beta x} (\phi^B - 4T^3 g^B)^2 d\mu dx + \frac{1}{2\xi} \int_0^B e^{2\beta x} |S_1|^2 dx.
            \end{align*}
            Combing the above inequality with \eqref{eq:cc/2}, \eqref{eq:cc/3} and \eqref{eq:cc/4}, the equality \eqref{eq:cc/1} becomes 
            \begin{align*}
                & \int_0^B \int_{-1}^1 (1-\xi - \mu\beta)e^{2\beta x}(\phi^B - 4T^3 g^B)^2 d\mu dx \nonumber\\
                &\quad +\frac12 \int_{-1}^0 |\mu| (\phi^B)^2(0,\cdot) d\mu  \le  \frac12 \int_0^1 \mu \phi_b^2 d\mu + \frac{1}{2\xi} \int_0^B e^{2\beta x} S_1^2 dx .
            \end{align*}
            Taking $\beta < 1$, then $1-\xi-\mu\beta \ge 1-\xi-\beta$. Taking $\xi=(1-\beta)/2$, the above inequality implies \eqref{eq:est-wB} and finishes the proof of Lemma \ref{lm:exbd}. 
        \end{proof}

\subsection{Existence on the half-space} \label{LMB3}
Next we pass to the limit $B\to\infty$ in system \eqref{eq:BBl1}-\eqref{eq:BBl2} and show the existence for system \eqref{eq:l-1}-\eqref{eq:l-2}.
\begin{proof}[Proof of Theorem \ref{thm2}, existence.]
    By the relation \eqref{eq:ccc/-}, we have 
    \begin{align}
        \int_0^B e^{2\beta x} |\partial_x g^B|^2 dx = \int_0^B e^{2\beta x} |\langle \mu \phi^B \rangle|^2 dx.
    \end{align}
    Using the fact $\langle \mu 4T^3g^B\rangle =0$ and H\"older's inequality, the above equation equals 
    \begin{align}\label{eq:el/00pre}
        \int_0^B e^{2\beta x} |\partial_x g^B|^2 dx &= \int_0^B e^{2\beta x}\left(\int_{-1}^1 \mu (\phi^B-4T^3g^B)d\mu \right)^2 dx \nonumber\\
        &\le \int_0^B \left(\int_{-1}^1 \mu d\mu \right)\left(\int_{-1}^1(\phi^B-4T^3 g^B)^2 d\mu\right) dx \nonumber\\
        &\le \frac23\int_0^B\int_{-1}^1 e^{2\beta x} (\phi^B-4T^3 g^B)^2 d\mu dx \nonumber\\
        &\le \frac23 \frac{1}{1-\beta} \int_0^1 \mu \phi_b^2 d\mu + \frac{4}{3(1-\beta)^2} \int_0^\infty e^{2\beta x} |S_1|^2 dx,
    \end{align}
    which is uniformly bounded. Hence 
    \begin{align}
        |g^B(x)| &= |g^B(x)-g^B(0)| = \left|\int_0^x \partial_z g^B(z)dz \right| \nonumber\\
        & = \left|\int_0^x e^{-\beta z} e^{\beta z} \partial_z g^B(z) dz \right| \nonumber\\
        &\le \left(\int_0^x e^{2\beta z} |\partial_z g|^2 dz \right)^{\frac12} \left(\int_0^x e^{-2\beta z} dz\right)^{\frac12} \nonumber\\
        &\le \frac{1}{\sqrt{2\beta}}e^{-\beta x}\left(\frac23 \frac{1}{1-\beta} \int_0^1 \mu \phi_b^2 d\mu + \frac{4}{3(1-\beta)^2} \int_0^\infty e^{2\beta x} |S_1|^2 dx\right)^\frac12 \nonumber\\
        &=N_{\beta} e^{-\beta x}.
    \end{align}
    Hence $g^B \in L_{\rm loc}^2(\mathbb{R}_+)$. By the equation \eqref{eq:BBl1}, 
    \begin{align}
        \int_0^B e^{2\beta x} |\partial_x^2 g^B|^2 dx \le \frac23 \int_0^B\int_{-1}^2 (\phi^B-4T^3g^B)^2 d\mu dx + 2\int_0^B e^{2\beta x} |S_1|^2 dx
    \end{align}
    and is also uniformly bounded. Hence $g^B\in H^2_{\rm loc}(\mathbb{R}_+)$ is uniformly bounded and we can find a subsequence that 
    \begin{equation}\label{eq:gconv/00}
    \begin{aligned}
        &g^B \rightharpoonup g,\quad \text{weakly in } H^2_{\rm loc}(\mathbb{R}_+),\\
        &g^B \to g,\quad \text{strongly in } C^1_{\rm loc}(\mathbb{R}_+).
    \end{aligned}
    \end{equation}
    Moreover, by the continuity of the trace operator, we can pass to the limit $B\to\infty$ in \eqref{eq:B/l1} and get that $g(0)=0$.

    By the above convergence result and the uniform boundness of $\int_0^B\int_{-1}^1 (\phi^B-4T^3g^B)^2 d\mu dx$, there exists subsequence such that 
    \begin{align}
        \phi^B - 4T^3 g^B \rightharpoonup \phi - 4T^3 g,\quad\text{weakly in } L^2_{\rm loc}(\mathbb{R}_+\times [-1,1]).
    \end{align}
    Hence 
    \begin{align}\label{eq:phiconv/00}
        \phi^B \rightharpoonup \phi,\quad \text{weakly in } L^2_{\rm loc}(\mathbb{R}_+\times [-1,1]).
    \end{align}
    Moreover, by equation \eqref{eq:BBl2},
    \begin{align}
        \int_0^B\int_{-1}^1 |\partial_x (\mu\psi)|^2 d\mu dx \le \int_0^B \int_{-1}^1 (\phi^B-4T^3 g^B)^2 d\mu dx + 2\int_0^B |S_1|^2 dx,
    \end{align}
    is uniformly bounded, hence $\partial_x (\mu\psi)\in L^2_{\rm loc}(\mathbb{R}_+\times [-1,1])$. Thus we can use the trace theorem and pass to the limit in \eqref{eq:B/l2} to get that $\phi(0,\mu)=\phi_b(\mu)$ for any $\mu>0$. Moreover, we can pass the weak limit in \eqref{eq:BBl1}-\eqref{eq:BBl2} and use \eqref{eq:gconv/00} and \eqref{eq:phiconv/00} and obtain that $(g,\phi)$ satisfies the system \eqref{eq:l-1}-\eqref{eq:l-2} with boundary conditions \eqref{eq:l-1b}-\eqref{eq:l-2b}.

    The proof is divided in to four steps. In the first step, we pass to the limit $B\to\infty$ in  \eqref{eq:B/l1}-\eqref{eq:B/l2} and prove the limit satisfies system \eqref{eq:l-1}-\eqref{eq:l-2} in the weak sense. Then we show the uniform estimate and weighted decay property \eqref{eq:thm2.decay}.
\subsection{Weighted estimate and exponentially decay}\label{LMB4}
Next we derive the uniform estimate for system \eqref{eq:l-1}-\eqref{eq:l-2} on the half-space and show the exponentially decay properties.
\begin{proof}[Proof of Theorem \ref{thm2}, weighted estimate and decay]
    The weak convergence of $\phi^B- 4(T^B)^3 g^B$ in $L^2_{\rm loc}(\mathbb{R}_{+})$ and the weak lower semi-continuity of the norm $\|\cdot\|_{L^2}$ implies that 
    \begin{align}
        \|e^{\beta x}(\phi-4T^3 g)\|_{L^2(\mathbb{R}_{+}\times [-1,1])}^2\le \liminf_{B\to\infty}\|e^{\beta x}(\phi^B - 4(T^B)^3 g^B) \|_{L^2([0,B]\times [-1,1])}^2.
    \end{align}
    Moreover, by the trace theorem we have 
    \begin{align}
       \lim_{B\to\infty} \int_{-1}^0|\mu| (\phi^B)^2(0,\cdot) d\mu = \int_{-1}^0 |\mu| \phi^2(0,\cdot) d\mu.
    \end{align}
    Therefore, we can take the $\liminf_{B\to\infty}$ in \eqref{eq:est-wB} and obtain 
    \begin{align}\label{eq:el/00}
        &\int_0^\infty \int_{-1}^1 e^{2\beta x}(\phi^B - 4(T^B)^3 g^B)^2 d\mu dx +\frac{1}{1-\beta} \int_{-1}^0 |\mu| \phi^2(0,\cdot) d\mu \nonumber\\
        &\qquad \le  \frac{1}{1-\beta} \int_0^1 \mu \phi_b^2 d\mu + \frac{2}{(1-\beta)^2} \int_0^\infty e^{2\beta x} |S_1|^2 dx,
    \end{align}
    which is the estimate \eqref{eq:el/2} in Theorem \ref{thm2}.

    To show the exponentially decay property, we first note that the condition \eqref{eq:ccc/-} still holds for $g,\phi$ on the half-line. We can pass to the limit in \eqref{eq:ccc/-} due to \eqref{eq:gconv/00} and \eqref{eq:phiconv/00}, and 
    \begin{align}\label{eq:---}
        \partial_x g(x) = \langle \mu \phi(x,\cdot)\rangle,\quad \text{ for } x \in \mathbb{R}_+ \quad \text{almost everywhere}.
    \end{align}
    In order to show the above relation still holds at infinity, we need to show this quantity are bounded.
    To show this, first comparing the equation \eqref{eq:l-1} to $\langle \eqref{eq:l-2}\rangle$ gives 
    \begin{align}
        \partial_x(\partial_x g - \langle \mu \phi\rangle) = 0.
    \end{align}
    By equation \eqref{eq:l-2} and the estimate \eqref{eq:el/00}, similarly as \eqref{eq:el/00pre},
    \begin{align}\label{eq:av/1}
        \int_0^\infty e^{2\beta x} |\langle \mu \phi \rangle|^2 dx &\le \frac23 \int_0^\infty\int_{-1}^1 (\phi-4T^3 g)^2 d\mu dx \nonumber\\
        &\le  \frac23\frac{1}{1-\beta} \int_0^1 \mu \phi_b^2 d\mu + \frac{4}{3(1-\beta)^2} \int_0^\infty e^{2\beta x} |S_1|^2 dx.
    \end{align}
    Moreover, we have 
    \begin{align}
        \partial_x \langle \mu \phi\rangle = -\langle \phi-4T^3 g\rangle + \langle S_1\rangle,
    \end{align} 
    and so 
    \begin{align}\label{eq:av/2}
        &\int_0^\infty e^{2\beta x}|\partial_x \langle \mu\phi\rangle|^2 dx\nonumber\\
         &\quad\le \int_0^\infty e^{2\beta x}|\langle \phi-4T^3 g\rangle|^2 dx + \int_0^\infty e^{2\beta x}|\langle S_1\rangle|^2 dx \nonumber\\
        &\quad\le 2 \int_0^\infty\int_{-1}^1 e^{2\beta x}(\phi-4T^3 g)^2 dx + 2\int_0^\infty e^{2\beta x}|S_1|^2 dx \nonumber\\
        &\quad\le 2\int_0^\infty e^{2\beta x}|S_1|^2 dx+  \frac{2}{1-\beta} \int_0^1 \mu \phi_b^2 d\mu + \frac{4}{(1-\beta)^2} \int_0^\infty e^{2\beta x} |S_1|^2 dx
    \end{align}
    From \eqref{eq:av/1} and \eqref{eq:av/2}, we can apply the Barbalat's lemma \cite{farkas2016variations} and deduce
    \[
    \langle \mu \phi(x,\cdot) \rangle \to 0, \quad \text{as} \quad x\to \infty
    \]
    Hence 
    \begin{align}
        | \langle \mu \phi(x,\cdot) \rangle| = \left|\int_x^\infty e^{-\beta x} e^{\beta x}\partial_x \langle \mu\phi\rangle dx\right| \le \frac{1}{\sqrt{2\beta}} e^{-\beta x}\left(\int_0^\infty e^{2\beta x}|\partial_x \langle \mu\phi\rangle|^2 dx\right)^\frac12,
    \end{align}
    and is uniformly bounded. Thus \eqref{eq:---} holds for all $x\in [0,\infty]$. Using the relation \eqref{eq:---} and \eqref{eq:av/1}, we obtain 
    \begin{align}
        \int_0^\infty e^{2\beta x}|\partial_x g|^2 dx \le  \frac23\frac{1}{1-\beta} \int_0^1 \mu \phi_b^2 d\mu + \frac{4}{3(1-\beta)^2} \int_0^\infty e^{2\beta x} |S_1|^2 dx = 2\beta N_{\beta}^2.
    \end{align}
    Hence 
    \begin{align}
       | g(x)| = \left|\int_0^x \partial_z g(z) dz\right| &= \left|\int_0^x   e^{-\beta x} e^{\beta x}\partial_x g dx \right| \nonumber\\
       &\le \frac{1}{\sqrt{2\beta}} \left(\int_0^\infty e^{2\beta x}|\partial_x g|^2 dx\right)^\frac12 = N_{\beta}.
    \end{align}
    We can let $x\to \infty$ and obtain $g_\infty:=\lim_{x\to\infty} g(x)$ exists and 
    \begin{align}
        |g_\infty|\le N_{\beta}.
    \end{align}
    We also have 
    \begin{align}
        |g(x)-g_\infty| &=  \left|\int_x^\infty   e^{-\beta x} e^{\beta x}\partial_x g dx \right| \nonumber\\
        &\le \frac{1}{\sqrt{2\beta}} e^{-\beta x} \left(\int_0^\infty e^{2\beta x}|\partial_x g|^2 dx\right)^\frac12 = N_{\beta} e^{-\beta x},
    \end{align}
    which gives the first inequality in \eqref{eq:thm2.decay}. The other two inequalities in \eqref{eq:thm2.decay} can be derived by using the formula \eqref{eq:sol1/mu+}-\eqref{eq:sol1/mu-} in a similar manner as \eqref{eq:dec1}-\eqref{eq:dec2}. For $\mu>0$,
    \begin{align}\label{eq:dec11}
        &\left|\psi(x,\mu) - 4T_\infty^3 g_\infty - \int_0^x \frac{1}{\mu} e^{-\frac{x-s}{\mu}} S_1(s) ds\right|\\
         &= \left|\psi_b e^{-\frac{x}{\mu}} + \int_0^x \frac{1}{\mu} e^{-\frac{x-s}{\mu}} (4T^3(s) g(s) + S_1(s)) ds - 4 T_\infty^3 g_\infty \right| \nonumber\\
        & = \left|\psi_b e^{-\frac{x}{\mu}} + \int_0^x \frac{1}{\mu} e^{-\frac{x-s}{\mu}}(4T^3(s)g(s) - T_\infty^4) ds - 4T^3_\infty g_\infty e^{-\frac{x}{\mu}} ds\right| \nonumber\\
        &\le  |\psi_b-4T_\infty^3 g_\infty| e^{-\frac{x}{\mu}} + \int_0^x \frac{1}{\mu} e^{-\frac{x-s}{\mu}}|4T^3(s)g(s) - 4T_\infty^3g_\infty| ds  \nonumber \\
        &\le |\psi_b - 4T_\infty^3 g_\infty| e^{-\frac{x}{\mu}} + \int_0^x \frac{1}{\mu} e^{-\frac{x-s}{\mu}} 4|T^3(s)-T_\infty^3|g(s) ds \nonumber\\
        &\quad + \int_0^x \frac{1}{\mu} e^{-\frac{x-s}{\mu}} 4T_\infty^3|g(s)-g_\infty| ds \nonumber\\
        &\le  |\psi_b - 4T_\infty^3g_\infty| e^{-\frac{x}{\mu}} + \int_0^x \frac{1}{\mu} e^{-\frac{x-s}{\mu}} 4\cdot 3(T_b+2M_{\alpha})^2M_{\alpha}e^{-\alpha x} N_{\beta} ds \nonumber\\
        &\quad + \int_0^x \frac{1}{\mu} e^{-\frac{x-s}{\mu}} 4(T_b+M_{\alpha})^3 N_{\beta} e^{-\beta x} dx \nonumber\\
        &= |\psi_b - 4T_\infty^3g_\infty| e^{-\frac{x}{\mu}}  + 12(T_b+2M_{\alpha})^2M_{\alpha} N_{\beta} \frac{1}{1-\mu \alpha}(e^{-\alpha x} - e^{-\frac{x}{\mu}}) \nonumber\\
        &\quad + 4(T_b+M_{\alpha})^3N_{\beta} \frac{1}{1-\mu\beta} (e^{-\beta x}-e^{-\frac{x}{\mu}}) \nonumber\\
        &\le |\psi_b - 4T_\infty^3 g_\infty| e^{-\frac{x}{\mu}} + 4(T_b+2M_{\alpha})^2(4M_{\alpha}+T_b)N_{\beta} e^{-\beta x},        
    \end{align}
    where in the last inequality we take $\alpha=\beta$. For $\mu<0$,
    \begin{align}\label{eq:dec12}
        &\left| \psi(x,\mu) - 4T_\infty^3 g_\infty + \int_x^\infty \frac{1}{\mu} e^{\frac{x-s}{\mu}} S_1(s) ds\right| \nonumber\\
        &= \left|- \int_x^\infty \frac{1}{\mu} e^{-\frac{x-s}{\mu}}4T^3(s)g(s) ds - 4T_\infty^3g_\infty \right|\nonumber\\
        & = \left|- \int_x^\infty \frac{1}{\mu} e^{-\frac{x-s}{\mu}}4T^3(s)g(s) ds  + \int_x^\infty \frac{1}{\mu} 
        e^{-\frac{x-s}{\mu}} 4T^3_\infty g_\infty ds\right| \nonumber\\
        &\le -\int_x^\infty \frac{1}{\mu} 
        e^{-\frac{x-s}{\mu}}  |4T^3(s)g(s) - 4T_\infty^3g_\infty| ds \nonumber \\
        &\le -\int_x^\infty \frac{1}{\mu} e^{-\frac{x-s}{\mu}} 4|T^3(s)-T_\infty^3|g(s) ds - \int_x^\infty \frac{1}{\mu} e^{-\frac{x-s}{\mu}} 4T_\infty^3|g(s)-g_\infty| ds \nonumber\\     
        &\le -\int_x^\infty \frac{1}{\mu} e^{-\frac{x-s}{\mu}} 4\cdot 3(T_b+2M_{\alpha})^2M_{\alpha}e^{-\alpha x} N_{\beta} ds \nonumber\\&\quad- \int_x^\infty \frac{1}{\mu} e^{-\frac{x-s}{\mu}} 4(T_b+M_{\alpha})^3 N_{\beta} e^{-\beta x} dx \nonumber\\
        &= 12(T_b+2M_{\alpha})^2M_{\alpha} N_{\beta} \frac{1}{1-\mu \alpha}e^{-\alpha x} + 4(T_b+M_{\alpha})^3N_{\beta} \frac{1}{1-\mu\beta} e^{-\beta x} \nonumber\\
        &= -4(T_b+2M_{\alpha})^3 M_{\alpha}\frac{1}{-1+\mu\alpha} e^{-\alpha x} \nonumber\\
        &\le  4(T_b+2M_{\alpha})^2(4M_{\alpha}+T_b)N_{\beta} e^{-\beta x},
    \end{align}
    where in the last inequality we also take $\alpha=\beta$. Since for $\mu>0$
    \begin{align}
        \left|\int_0^x \frac{1}{\mu} e^{-\frac{x-s}{\mu}} S_1(s) ds\right|&= \left|\int_0^x \frac{1}{\mu} e^{-\frac{x-s}{\mu}} e^{-\beta s} e^{\beta s} S_1(s) ds\right| \nonumber\\
        &\le \left(\int_0^x \frac{1}{\mu^2} e^{-\frac{2(x-s)}{\mu}} e^{-2\beta s} ds\right)^{\frac12} \left(\int_0^x e^{2\beta s}|S_1(s)|^2 ds\right)^\frac12 \nonumber\\
        &\le \frac{1}{1-2\mu\beta} (e^{-2\beta x} - e^{-2\frac{x}{\mu}})^\frac12 \|e^{\beta x}S_1\|_{L^2(\mathbb{R}_+)} \nonumber\\
        &\le \frac{1}{1-2\mu\beta}\|e^{\beta x}S_1\|_{L^2(\mathbb{R}_+)}  e^{-\beta x},
    \end{align}
    and similarly for $\mu<0$,
    \begin{align}
        \left|\int_x^\infty \frac{1}{\mu} e^{-\frac{x-s}{\mu}} S_1(s) ds\right|&= \left|\int_x^\infty \frac{1}{\mu} e^{-\frac{x-s}{\mu}} e^{-\beta s} e^{\beta s} S_1(s) ds\right| \nonumber\\
        &\le \left(\int_x^\infty \frac{1}{\mu^2} e^{-\frac{2(x-s)}{\mu}} e^{-2\beta s} ds\right)^{\frac12} \left(\int_x^\infty e^{2\beta s}|S_1(s)|^2 ds\right)^\frac12 \nonumber\\
        &\le \frac{1}{1-2\mu\beta} e^{-\beta x}  \|e^{\beta x}S_1\|_{L^2(\mathbb{R}_+)}.
    \end{align}
    Combing the above inequalities with \eqref{eq:dec11} and \eqref{eq:dec12}, we obtain the last two inequalities in \eqref{eq:thm2.decay}.
\end{proof}

\subsection{Uniqueness}\label{LMB5}
Suppose $(g_1,\phi_1)$ and $(g_2,\phi_2)$ are two solutions of the system \eqref{eq:l-1}-\eqref{eq:l-2} with boundary conditions \eqref{eq:l-1b}-\eqref{eq:l-2b}. Then $h=g_1-g_2$ and $\varphi = \phi_1-\phi_2$ satisfy the system 
\begin{align}
 \partial_x^2 h + \langle \varphi - 4T^3 h \rangle &=0, \label{eq:f/1}\\
 \mu \partial_x \varphi + (\varphi - 4T^3 h) &= 0,\label{eq:f/2}
\end{align}
and the boundary conditions 
\begin{align}\label{eq:ff1}
 h(0) = 0, \quad \varphi(0,\mu) = 0,\text{ for }\mu>0.
\end{align}
From the previous section, we have that $h,\varphi$ also verify 
\begin{align}
 \lim_{x\to\infty}\partial_x h = 0, \lim_{x\to\infty}\langle \mu \varphi(x,\cdot)\rangle = 0,\label{eq:ff2}
\end{align}
and  
\begin{align}
 \partial_x h(x) = \langle \mu \varphi(x,\cdot) \rangle, \quad\text{for any} \quad x\in \mathbb{R}_{+}.\label{eq:ff3}
\end{align}

We now multiply \eqref{eq:f/1} by $4T^3 h$ and \eqref{eq:f/2} by $\varphi$ and integrate to get 
\begin{align*}
 -\int_0^\infty 4T^3 h \partial_x^2 h dx + \int_0^\infty \mu \partial_x \frac{\varphi^2}{2} d\mu dx + \int_0^\infty \int_{-1}^1 (\varphi-4T^3 h)^2 d\mu dx =0.
\end{align*}
We  integrate by parts and by using \eqref{eq:ff1}-\eqref{eq:ff2} and  the inequality \eqref{sp2} due to the spectral assumption \ref{asA}, we obtain
\begin{align*}
 -\int_0^\infty 4&T^3 h \partial_x^2 h dx + \int_0^\infty \mu \partial_x \frac{\varphi^2}{2} d\mu dx \nonumber\\
 =~&- 4T^3 h \partial_x h \bigg|_0^\infty + \left(\int_0^\infty 4T^3 |\partial_x h|^2 dx + \frac{1}{2} \int_0^\infty \partial_x(4T^3) \partial_x g^2 dx\right)+ \frac{1}{2}\langle \mu \varphi^2\rangle \bigg|_0^\infty \nonumber   \\
 =~& -\frac{1}{2} \langle \mu(\varphi(0,\cdot)-4T^3 h(0))^2 dx + \left(\int_0^\infty 4T^3 |\partial_x h|^2  + \frac{1}{2} \int_0^\infty \partial_x(4T^3) \partial_x g^2 dx\right)  \nonumber\\
 \ge~& \frac{1}{2}\int_{-1}^0 |\mu|(\varphi(0,\cdot)|^2 d\mu.
\end{align*}
The relation \eqref{eq:ff3} is used in the second equality. 
Taking the above inequality into the previous equation leads to 
\begin{align*}
 \frac{1}{2}\int_{-1}^0 |\mu|(\varphi(0,\cdot)|^2 d\mu + \int_0^\infty \int_{-1}^1 (\varphi-4T^3 h)^2 d\mu dx \le 0.
\end{align*}
This implies that 
\begin{align*}
 \varphi - 4T^3 h = 0, \varphi(0,\mu) =0,\text{ for any }\mu\in [-1,1],
\end{align*}
and so 
\begin{align*}
 \partial_x^2 h = 0, \partial_x \langle \mu \varphi \rangle =0, \langle \mu \varphi(0,\cdot) \rangle =0.
\end{align*}
Therefore, we obtain 
\begin{align*}
 \langle \mu \varphi(x,\cdot) \rangle = 0, \text{ for any }x\ge 0
\end{align*}
and 
\begin{align*}
 \partial_x h (x)= \langle \mu \varphi(x,\cdot)\rangle =0, \quad \text{for any}\quad x\ge 0.
\end{align*}
 Consequently, we have  $h \equiv 0$ and $\varphi \equiv 4T^3 h \equiv 0 $. Finally,  system \eqref{eq:f/1}-\eqref{eq:f/2} has only zero solutions. Therefore, the solution to system \eqref{eq:l-1}-\eqref{eq:l-2} is unique.

\end{proof}

\section{Uniqueness of the nonlinear Milne problem}\label{proofThm3}
In this section, we prove Theorem \ref{thm3}. We first show the uniqueness for small boundary data and then for small perturbations in a neighborhood of functions satisfying the spectral assumptions.

\subsection{Sufficient condition for the spectral assumption} \label{51}
We next give some sufficient and necessary conditions for the spectral assumption \ref{asA} using the Hardy's inequality.
   \begin{lemma}\label{lem.spcond}
    The spectral assumption \ref{asA} holds for function $T\in C^1(\mathbb{R}_+)$ if 
    \begin{align}\label{eq:Acond}
        A_0 := \sup_{r\in (0,\infty)} \left(\int_r^\infty e^{2\beta x} 36 T|\partial_x T|^2 dx \right)^{\frac12} \left(\int_0^r \frac{1}{e^{2\beta x}4T^3 } dx\right)^{\frac12} < \infty
    \end{align}
    satisfies $A_0 < \frac12$ for some constant $\beta>0$. Moreover, if the spectral assumption \ref{asA} holds, then $A_0<1$.
    Furthermore, if $T$ satisfies $T_m\le |T| \le C$, $|\partial_x T| \le C e^{-\alpha x}$ for some constants $T_m>0$, $C>0$ and $\alpha>0$. Then the spectral assumption \ref{asA} holds if 
    \begin{align}
    A_1 := \sup_{r\in (0,\infty)} \left(\int_r^\infty  36 T|\partial_x T|^2 dx \right)^{\frac12} \left(\int_0^r \frac{1}{4T^3 } dx\right)^{\frac12} < \frac12 .
    \end{align}
\end{lemma}
\begin{proof}
    This lemma is a direct consequence of the the generalized Hardy's inequality (Lemma \ref{lm:hardy} in Appendix \ref{Hardyinequality}). We take $u(x) = e^{2\beta x} 36 T|\partial_x T|^2$, $v(x)=e^{2\beta x} 4T^3$, $F(x) = \partial_x f$ and $b=\infty$, $p=p'=2$ in \eqref{eq:hardy} and get that
    \begin{align}\label{eq:lm5/est}
        \int_0^\infty f^2 e^{2\beta x}  36 T|\partial_x T|^2 dx \le M_0 \int_0^\infty |\partial_x f|^2 e^{2\beta x} 4T^3 dx,
    \end{align}
    holds for any $f \in C^1(\mathbb{R}_+)$ with $f(0)=0$,
    with the best constant $M_0$ satisfying $A_0^2\le M_0 \le 4A_0^2$, where $A_0$ is given by \eqref{eq:Acond}.
    
    
    On the other hand, if the spectral assumption \ref{asA} holds, then the best constant $M_0$ in \eqref{eq:lm5/est} satisfies $M_0<1$. Since $A_0^2 \le M_0$ due to Lemma \ref{Hardyinequality}, $A_0<1$.

    Finally, if $|T|\le C$, $|\partial_x T| \le Ce^{-\alpha x}$, then         
    \begin{align}
    T(x)|\partial_x T(x)|^2 \le C e^{-\alpha x}.
    \end{align}
    This implies 
    \begin{align}
    \int_r^\infty (e^{2\beta x}-1) T|\partial_x T|^2 dx &\le C \int_r^\infty  (e^{2\beta x}-1) e^{-\alpha x} dx = C e^{-\alpha r}\frac{2\beta-\alpha(1-e^{2\beta r})}{\alpha(\alpha-2\beta)} \nonumber\\
    &\le C \frac{1}{ (\alpha-2\beta)} e^{-(\alpha-2\beta)r}-C \frac{1}{\alpha} e^{-\alpha r} 
    \end{align}
    for $\beta<\alpha/2$.
    
    Assume $T\ge T_m>0$ for some constant $T_m>0$. Then 
    \begin{align}
    \int_0^r \frac{1}{4T^3} dx \le \frac{1}{4T_m^3} r 
    \end{align}
    Therefore,
    \begin{align}
    &\int_r^\infty e^{2\beta x} 36 T|\partial_x T|^2 dx \int_0^r \frac{1}{e^{2\beta x}4T^3 } dx \\
    &\quad \le  \int_r^\infty e^{2\beta x} 36 T|\partial_x T|^2 dx \int_0^r \frac{1}{4T^3 } dx \nonumber\\
    &\quad =\int_r^\infty (e^{2\beta x}-1) 36 T|\partial_x T|^2 dx \int_0^r \frac{1}{4T^3 } dx + \int_r^\infty e^{2\beta x} 36 T|\partial_x T|^2 dx \int_0^r \frac{1}{4T^3 } dx \nonumber\\
    &\quad \le \frac{1}{4T_m^3} r(C \frac{1}{ (\alpha-2\beta)} e^{-(\alpha-2\beta)r}-C \frac{1}{\alpha} e^{-\alpha r} ) + A_1^2.
    \end{align}
    for $\beta$ sufficiently small, the first term on the last line of the above inequality can be sufficiently small and so 
    \begin{align}
    A_0 < A_1 + \eps  <\frac12.
    \end{align}
    Hence by the previous argument,
    the spectral assumption \ref{asA} holds.

\end{proof}

\subsection{Uniqueness for almost  well-prepared boundary data}\label{52}
    We first show that the spectral assumption is not empty and its holds when $\int_0^1 (\psi_b-T_b^4)^2 d\mu$ is small using the Hardy's inequality, [Appendix \ref{Hardyinequality}]. We then give the uniqueness  of the solution for system \eqref{eq:m1}-\eqref{eq:m2}.

    \begin{lemma}\label{lm:wp}
        Assume the boundary data $(T_b,\psi_b)$ satisfies \eqref{eq:ass2} for some constant $C_b$ satisfying \eqref{eq:Cb}. Let $(T,\psi)$ be the corresponding solution obtained in Theorem \ref{thm:ex}. Then $T$ satisfies the inequality \eqref{eq:A2}. As a consequence, $T$ also satisfies the spectral assumption \ref{asA}.
    \end{lemma}
    \begin{proof}
    By Lemma \ref{lem.spcond}, the spectral assumption \ref{asA} holds if the condition \eqref{eq:Acond} is fulfilled. We next show \eqref{eq:Acond} holds when $\int_0^1 \mu(\psi_b-T_b^4)^2 d\mu$ is sufficiently small. Define $C_b :=\tfrac12 \int_0^1 \mu(\psi_b- T_b^4)^2 d\mu$, 
    \begin{align}\label{eq:5/1}
        &|\partial_x T | \le \frac{1}{\sqrt{2\alpha}} e^{-\alpha x} C_b^{\frac12}.
    \end{align}
     Theorem \ref{thm:ex} implies  that $(T,\psi)$ are uniformly bounded such that
    \begin{align*}
       0\le  T(x) \le \gamma,\quad 0 \le \psi(x) \le \gamma^4,
    \end{align*}
    for any $x\ge 0$. Hence for $\beta < \alpha$,
    \begin{align}\label{lm91}
        \int_r^\infty e^{2\beta x} 36 T|\partial_x T|^2 dx \le \frac{36}{2(\alpha-\beta) \sqrt{2\alpha}}\gamma C_b e^{-2(\alpha-\beta) r}.
    \end{align}

    On the other hand, due to \eqref{eq:pt2}, we get
    \begin{align*}
        |T(x) - T_b| &= \left|\int_0^x \partial_x T dx \right| = \left|\int_0^x e^{-\alpha z} e^{\alpha z} \partial_z T dx \right| \\
        & \le \left(\int_0^x e^{-2\alpha z} dx\right)^{\frac12} \left(\int_0^x e^{2\alpha z} |\partial_z T|^2 dx\right)^{\frac12} \\
        &\le \frac{1}{\sqrt{2\alpha}} \sqrt{1-e^{-2\alpha x}} \sqrt{\frac{2}{3(1-\alpha)}}C_b \\
        &\le \sqrt{\frac{1-e^{-2\alpha x}}{3\alpha(1-\alpha)}} C_b,
    \end{align*}
    and so
    \begin{align*}
        T(x) \ge T_b - \sqrt{\frac{1-e^{-2\alpha r}}{3\alpha(1-\alpha)}} C_b,
        \quad \text{ for any } 0\le x\le r,
    \end{align*}
    if $C_b$ is sufficiently small. For example, we can take
    $C_b = \tfrac12 T_b \sqrt{3\alpha(1-\alpha)}$
     and the above inequality implies 
    \begin{align}\label{t>tb}
        T(x) \ge \frac{1}{2} T_b.
    \end{align}
    Hence 
    \begin{align}\label{eq:318}
        \int_0^r \frac{1}{{e^{2\beta x}4T^3}} dx \le {\frac{1}{\beta T_b^3}} .
    \end{align} 
    Combining this inequality with \eqref{lm91} gives 
    \begin{align*}
        \int_r^\infty e^{2\beta x}36 T|\partial_x T|^2 dx \cdot \int_0^r \frac{1}{{e^{2\beta x} 4T^3}} dx \le  \frac{36}{2(\alpha-\beta)  \sqrt{2\alpha}}\gamma C_b {\frac{1}{\beta T_b^3}} e^{-2(\alpha-\beta) r}. 
    \end{align*}
    Since $e^{-2(\alpha-\beta)r} \le 1$, inequality \eqref{eq:Acond} is realized if 
    \begin{align}
         \frac{36}{2(\alpha-\beta)  \sqrt{2\alpha}}\gamma C_b {\frac{1}{\beta T_b^3}} < \frac14.
    \end{align}
    Therefore, inequality \eqref{eq:A2} is fulfilled if we take
    \begin{align}\label{eq:Cb}
        C_b = \min \left\{\frac12 T_b \sqrt{3\alpha(1-\alpha)}, \frac{1}{72} \gamma^{-1} T_b^3 \beta (\alpha-\beta)\sqrt{\alpha}
     \right\}
    \end{align} 
    in \eqref{eq:ass2}. By Lemma \ref{lem.spcond}, the spectral assumption \ref{asA} also holds.
\end{proof}
With this lemma, we can prove Theorem \ref{thm3} for the case when boundary data is almost  well-prepared.

\begin{proof}[Proof of Theorem \ref{thm3}, near well-prepared boundary data]

    Let $\left( T ,\psi \right)$ and $\left(T_1 ,\psi_1\right)$ be two solutions to system \eqref{eq:m1}-\eqref{eq:m2} with boundary conditions \eqref{eq:bd1}-\eqref{eq:bd2} satisfying the assumption of Theorem \ref{thm3}. Then $h:=T_1-T,\phi=\psi_1-\psi$ satisfies 
    \begin{align}
        &\partial_x^2 h + \langle \phi - w h \rangle = 0,\label{eq:li/1}\\
        &\mu \partial_x \phi + \phi - w h=0,\label{eq:li/2}
    \end{align}
    with boundary conditions 
    \begin{align*}
        h(0)=0, \phi(0,\mu)=0,\quad \text{for }\mu >0,
    \end{align*}
    where $w=(T_1^4-T^4)/(T_1-T).$ 
    By Lemma \ref{lm:wp}, $T$ and $T_1$ both satisfy the spectral assumption \ref{asA} and thus inequality \eqref{sp2} is satisfied for both $T$ and $T_1$.
    Moreover, one can show in the same way that $T_1^2T$(or $T_1^2 T$) also satisfies the inequality  
    \begin{align}\label{eq:sp/1}
        \int_0^\infty e^{2\beta x} 4T^2 T_1 |\partial_x f|^2 dx \ge \int_0^\infty e^{2\beta x} \frac{1}{4T^2 T_1} |\partial_x(4T^2T_1)|^2 f^2 dx
    \end{align}
    which holds if 
    \begin{align*}
        \int_0^\infty e^{2\beta x} \left(32T_2|\partial_x T|^2 + 8\frac{T^2}{T_1}|\partial_x T_1|^2\right) f^2 dx \le  \int_0^\infty e^{2\beta x} 4  T^2 T_1 |\partial_x g|^2 dx.
    \end{align*}
    According to the Hardy's inequality \eqref{eq:hardy}, a sufficient condition for the above inequality to hold is (see Lemma \ref{lem.spcond})
    \begin{align} \label{cond/r}
        \sup_{r\in (0,\infty)}\left(\int_r^\infty e^{2\beta x}\left(32T_1|\partial_x T|^2 + 8\frac{T^2}{T_1}|\partial_x T_1|^2\right) dx \cdot \int_0^r e^{2\beta x} \frac{1}{{4T^2T_1} } dx\right) \le \frac{1}{4}.
    \end{align}
    Taking $C_b = \tfrac12 T_b \sqrt{3\alpha(1-\alpha)}$, \eqref{t>tb} implies $T(x), T_1(x) \ge \tfrac12 T_b$. Due to the fact that $T,T_1\le \gamma$, we get 
    \begin{align*}
        &\int_r^\infty e^{2\beta x} \left(32T_1|\partial_x T|^2 + 8\frac{T^2}{T_1}|\partial_x T_1|^2\right) dx \nonumber\\
        &\quad \le \frac{32}{2(\alpha-\beta) \sqrt{2\alpha}}\gamma C_b e^{-2(\alpha-\beta) r} + \frac{8}{(\alpha-\beta)\sqrt{2\alpha} T_b} \gamma^2 C_b e^{-2(\alpha-\beta) r}.
    \end{align*}
 On the other hand,  $T(x), \, T_1(x)\ge \frac12 T_b$  implies 
    \begin{align*}
        \int_0^r \frac{1}{{4T^2 T_1}e^{2\beta x}} dx \le {\frac{1}{\beta T_b^3}} .
    \end{align*}
    Hence \eqref{eq:318} is also realized. Then
    \begin{align}
        &\int_r^\infty e^{2\beta x} \left(32T_1|\partial_x T|^2 + 8\frac{T^2}{T_1}|\partial_x T_1|^2\right) dx \cdot \int_0^r  \frac{1}{{4T^2T_1e^{2\beta x}} } dx \nonumber\\
        &\quad \le \left(\frac{32}{2(\alpha-\beta) \sqrt{2\alpha}}\gamma  + \frac{8}{(\alpha-\beta)\sqrt{2\alpha} T_b} \gamma^2\right) \frac{C_b}{\beta T_b^3} e^{-2(\alpha-\beta) r} \nonumber\\
        &\quad \le \left(\frac{32}{2(\alpha-\beta) \sqrt{2\alpha}}\gamma  + \frac{8}{(\alpha-\beta)\sqrt{2\alpha} T_b} \gamma^2\right) \frac{C_b}{\beta T_b^3}.
            \end{align}
    Hence \eqref{cond/r} is fulfilled if 
    \begin{align*}
        C_b \le  \min \bigg\{&\frac12 T_b \sqrt{3\alpha(1-\alpha)}, \frac{1}{72} \gamma^{-1} T_b^3 \beta (\alpha-\beta)\sqrt{\alpha},\nonumber\\
        & \frac14 \sqrt{\alpha}(\alpha-\beta) T_b^{3}  \gamma^{-1}(32\gamma+16\gamma T_b^{-1})^{-1}  \bigg\}.
    \end{align*}
    Then \eqref{eq:sp/1} holds 
    and $w=(T_1^4-T^4)/(T_1-T)$ satisfies the spectral assumption \ref{asA}. 
    Moreover, by the proof of Lemma \ref{lem.spcond}, $M\le 4  \left(\frac{32}{2(\alpha-\beta) \sqrt{2\alpha}}\gamma  + \frac{8}{(\alpha-\beta)\sqrt{2\alpha} T_b} \gamma^2\right) \frac{C_b}{\beta T_b^3}$. With a smaller $C_b$, the constant $M$ in \eqref{eq:spassump2} of the spectral assumption \ref{asA} is expected to the smaller.
    
    Therefore, by Theorem \ref{thm2},
    the solution to the  linear problem \eqref{eq:li/1}-\eqref{eq:li/2} is unique. Consequently, we have $h=0,\psi=0$, which is $T_1=T,\psi_1=\psi$, i.e. the solution to system \eqref{eq:m1}-\eqref{eq:m2} is unique.

    \end{proof}
\subsection{Uniqueness under small perturbations}\label{smalS}
First we show that the spectral assumption holds under small perturbations.
\begin{lemma}\label{lm:sm}
    Assume $T$ satisfies the inequality \eqref{eq:A2}. Let $\mathcal{V}$ be the function space defined in \eqref{eq:spV}. Then there exists a small constant $\eps>0$ such that the inequality \eqref{eq:A2} holds for any $T_1\in \mathcal{V}$. As a consequence, the spectral assumption \ref{asA} is realized for any $T_1\in\mathcal{V}$.
\end{lemma}
\begin{proof}
Let $A_0$ be given by \eqref{eq:Acond}. Let $h=(T_1-T)/\eps$.
We then calculate
\begin{align}
    &\int_r^\infty e^{2\beta x}36 T_1 |\partial_x T_1|^2 dx \cdot \int_0^r \frac{1}{4T_1^3 e^{2\beta x}} dx \\
    &\quad =\int_r^\infty e^{2\beta x} 36 (T+\varepsilon h)|\partial_x T + \varepsilon \partial_x h|^2 dx \cdot \int_0^r  \frac{1}{4(T+\varepsilon h)^3e^{2\beta x}} dx \nonumber\\
    &\quad = \int_r^\infty  e^{2\beta x} 36 T |\partial_x T|^2 dx \cdot \int_0^r \frac{1}{4T^3  e^{2\beta x}} dx \nonumber\\
    &\qquad + \varepsilon\int_r^\infty  e^{2\beta x} 36  h |\partial_x T + \varepsilon \partial_x h|^2 dx \cdot \int_0^r \frac{1}{4(T+\varepsilon h)^3 e^{2\beta x}}dx \nonumber\\
    &\qquad +  \varepsilon \int_r^\infty  e^{2\beta x} 36 T |\partial_x T|^2 dx \cdot \int_0^r \frac{-3T^2 h - 3\varepsilon Th^2 - \varepsilon^2 h^3 }{4(T+\varepsilon h)^3 T^3 e^{2\beta x}} dx\nonumber\\
    &\qquad + \varepsilon \int_r^\infty  e^{2\beta x} 36 (2 \partial_x T \partial_x h + \varepsilon |\partial_x h|^2) dx \cdot \int_0^r \frac{1}{4(T+\varepsilon h)^3 e^{2\beta x}} dx.
\end{align}
Consequently, 
\begin{align}
    (A')^2&:=\sup_{r\in (0,\infty)}\left(\int_r^\infty  e^{2\beta x} 36 (T+\varepsilon h)|\partial_x T + \varepsilon \partial_x h|^2 dx \cdot \int_0^r \frac{1}{4(T+\varepsilon h)^3 e^{2\beta x}} dx\right) \nonumber\\
    &\le A_0^2 + C_1 \varepsilon.
\end{align}
for some constant $C_1>0$. By the Hardy's inequality, we have 
\begin{align}
    C_{hd}' \int_0^\infty  e^{2\beta x} 4(T+\varepsilon h)^3|\partial_x f|^2 dx \ge \int_0^\infty  e^{2\beta x}  36 (T+\varepsilon h) |\partial_x (T+\varepsilon h)|^2 f^2 dx
\end{align}
where $C_{hd}'\le 4 (A')^2 \le 4A_0^2 + 4C_1 \varepsilon$. Therefore if $A_0^2 <1/4$, then $(A')^2<1/4$ for $\eps$ sufficiently small, and the inequality \ref{eq:A2} holds for $T+\eps h$ if $T$ satisfies \eqref{eq:Acond} with $A_0<1/2$.
Therefore, we conclude that the inequality \eqref{eq:A2} holds for any $T_1 \in C^1(\mathbb{R}_+)$ such that $\|T_1-T\|_{H^1(\mathbb{R}_+)} \le C_2\varepsilon$ for some constant $C_2 >0$, where $T$ satisfies \eqref{eq:Acond} with $A_0<1/2$. 

Moreover, for a smaller $A_0$, $\eps$ could be larger in order for $4A_0^2 + 4C_1\eps <1$ to hold. A smaller $A_0$ implies a smaller $M$ in \eqref{eq:spassump2}. Therefore, it is expected that a smaller $M$ in the spectral assumption is more stable with respect to perturbations. We can also show when $\eps$ is large, the spectral assumption does not hold.  Assume $\varepsilon$ is large, then the dominant term is 
\begin{align}
    \int_r^\infty  e^{2\beta x} 36 h |\partial_x h|^2 dx \cdot \int_0^r \frac{1}{4 h^3 e^{2\beta x}} dx,
\end{align}
which can become large for $h\in L^2_{\rm loc}(\mathbb{R}_+)$.
 For example, we may take $\beta=0$, $h(x) = xe^{-x}$, the integral $\int_r^\infty 36 h |\partial_x h|^2 dx$ is finite positive while $\int_0^r \frac{1}{4h^3} dx = \int_0^r \frac{e^{3x}}{4x^3} dx$ diverges. Therefore, when $\varepsilon \to \infty$, $(A')^2 \to \infty$ and $C_{hd}' \ge (A')^2 $ also goes to $\infty$. Hence Lemma \ref{lm:sm} only holds for  $\eps$ sufficiently small.
\end{proof}

We proceed to prove Theorem \ref{thm3}.
\begin{proof}[Proof of Theorem \ref{thm3}: The case of small perturbations.]
Suppose $T, T_1$ are two functions satisfying the assumptions of Theorem \ref{thm3} and $T_1\in \mathcal{V}$ with $\mathcal{V}$ given in \eqref{eq:spV}. Setting $h=(T_1-T)/\eps$, then $w:=(T_1-T^4)/(T_1-T)$ is
\begin{align*}
    w &= \frac{(T+\eps h)^4-T^4}{\eps h} = 4T^3 + 6\eps T^2 h+ 4\eps^2 h^2 T + \eps^3 h^3
\end{align*}
is a small perturbation of $4T^3$, thus by the proof of Lemma \ref{lm:sm} and Lemma \ref{LemmaSA}, \eqref{sp2} still holds with $4T^3$ replaced by $w$. Thus by Theorem \ref{thm2}, $T_1=T$ and the solution is unique in $\mathcal{V}$. This finishes the proof of Theorem \ref{thm3}.

\end{proof}

     \appendix

     \section{Proof of Theorem \ref{thm2} for the general case}\label{sec:generalcase}
     Consider 
     \begin{align}
        &\partial_x^2 g + \langle \phi - 4T^3 g \rangle=\langle S_1 \rangle,\label{eq:ap--1}\\
        &\mu\partial_x \phi + (\phi - 4T^3 g ) = S_2,\label{eq:ap--2}
    \end{align}
    supplemented with the following boundary conditions
    \begin{align}
        &g(0) = g_b, \label{eq:ap--1b}\\
        &\phi(0,\mu) = \phi_b(\mu), \text{ for any } \mu \in(0,1]\label{eq:ap--2b},
    \end{align}
    Then the following theorem holds.
    \begin{theorem}
    Suppose $T$ satisfies the spectral assumption \ref{asA} for some constants $0<\beta<1$ and $M>0$. Assume $S_1=S_1(x,\beta),S_2=S_2(x,\beta)$ decays to zero exponentially such that $\|e^{\beta x}S_1\|_{L^2(\mathbb{R}_+\times[-1,1])},\|e^{\beta x}S_1\|_{L^2(\mathbb{R}_+\times[-1,1])}<\infty $.
    Then there exists a unique bounded solution $\left(g,\phi\right) \in L^2_{\rm loc}(\mathbb{R}_+)\times L^2_{\rm loc}(\mathbb{R}_{+}\times [-1,1])$
     to system \eqref{eq:ap--1}-\eqref{eq:ap--2} with boundary conditions \eqref{eq:ap--1b}-\eqref{eq:ap--2b},
     and the solution satisfies 
     \begin{align}
        |g(x)-g_\infty| \le C e^{-\beta x},\quad |\phi(x,\mu)-4T_\infty^3 g_\infty|\le C e^{-\beta x},
     \end{align}
     for some constant $C>0$ and $g_\infty \in \mathbb{R}$, where $T_\infty=\lim_{x\to\infty}T(x)$.

    \end{theorem}
    
    \begin{proof}
        First we can take $h = g-g_be^{-x}$ and system \eqref{eq:ap--1}-\eqref{eq:ap--2} becomes 
        \begin{align}
            &\partial_x^2 h + \langle \phi - 4T^3g\rangle = \langle S_1\rangle + \langle -\tfrac12 g_b^{-x}+4T^3g_be^{-x} \rangle =:\langle \tilde{S_1}\rangle,\\
            &\mu \partial_x\phi + \phi - 4T^3 g = S_2 + 4T^3 g_be^{-x}:=\tilde{S}_2,
        \end{align}
        with boundary conditions 
        \begin{align}
            g(0)=0,\quad  &\phi(0,\mu) = \phi_b(\mu), \text{ for any } \mu \in(0,1].
        \end{align}
        In order to solve the above problem, we fist construct $(g^1,\phi^1)$ by solving the following system 
        \begin{equation}\label{eq:so1}
        \begin{aligned}
            &\partial_x^2 g^1 + \langle \phi^1 - 4T^3 g^1\rangle = \langle \tilde{S}_1\rangle,\\
            &\mu\partial_x \phi^1 + \phi^1 - 4T^3 g^1 = \tilde{S}_2 - \frac12 \langle \tilde{S}_2-\tilde{S}_1\rangle,
        \end{aligned}
        \end{equation}
        with boundary conditions 
        \begin{align}
            g^1(0)=0,\quad \phi^1(0,\mu) = \phi_b(\mu),\quad \text{for any }\mu \in (0,1].
        \end{align}
        Next let $\varphi=\mu G$ with $G$ being the solution of 
        \begin{align}
            G(x) =- \frac32 \int_x^\infty \langle \tilde{S}_2 - \tilde{S}_1\rangle dx.
        \end{align}
        Then $\varphi$ satisfies 
        \begin{align}
            \langle \mu \partial_x \varphi + \mu \varphi\rangle = \langle \mu^2 \partial_x G\rangle = \frac32 \partial_x G = \langle \tilde{S}_2-\tilde{S}_1\rangle.
        \end{align}

        Next we construct $(g^2,\phi^2)$ by solving the system 
        \begin{equation}\label{eq:so2}
        \begin{aligned}
            &\partial_x^2 g^2 + \langle \phi^2 - 4T^3 g^2 \rangle = 0,\\
            &\mu\partial_x \phi^2 + \phi^2 - 4T^3 g^2 = -(\mu \partial_x \varphi + \mu \varphi )+ \frac12 \langle \tilde{S}_2-\tilde{S}_1\rangle,
        \end{aligned}
    \end{equation}
    with boundary conditions 
    \begin{align}
        g(0)=0, \quad \phi(0,\mu)=-\mu G(0),\quad \text{for any }\mu \in (0,1].
    \end{align}
        Then one can show directly that $(g=h_be^{-x} + g^1 + g^2,\phi = \mu \varphi +\phi^1 + \phi^2)$ satisfies system \eqref{eq:ap--1}-\eqref{eq:ap--2} with boundary conditions \eqref{eq:ap--1b}-\eqref{eq:ap--2b}.

        We can also obtain the decay estimate using Theorem \ref{thm2}. First, since $\langle \tilde{S}_1\rangle = \langle \tilde{S}_2 - \tfrac12 \langle \tilde{S}_2-\tilde{S}_1\rangle \rangle$, we can apply Theorem \ref{thm2} to get the existence and uniqueness of solutions to system \eqref{eq:so1}. Moreover, we can get from estimate  \eqref{eq:thm2.decay}
        \begin{align}
            &|g^1(x)-g^1_{\infty}| \le N_{\beta}^1 e^{-\beta x}, \\
            &|\phi^1(x,\mu) - 4T^3_\infty g_\infty| \le C_1(\beta,\psi_b,T_b,M_{\alpha},N_{\beta}^1) e^{-\beta x},\quad \text{for any }\mu \in [-1,1],
        \end{align}
        with 
        \begin{align}
            N_{\beta}^1=\frac{1}{\sqrt{2\beta}}\left(\frac{2}{3(1-\beta)}\int_0^1\mu\phi_b^2 d\mu + \frac{2}{3(1-\beta)^2} \|e^{\beta x} (\tilde{S}_2)-\tfrac12 \langle \tilde{S}_2-\tilde{S}_1\rangle\|_{L^2(\mathbb{R}_+\times [-1,1])}^2 \right)^{\frac12}.
        \end{align}
        Moreover,
        \begin{align}
            |g^1(x)| \le 2N_{\beta}^1,\quad |\phi^1(x,\mu)|\le C(\beta,\psi_b,T_b,M_{\alpha},N_{\beta}^1).
        \end{align}
        Similarly, we can also get the existence and uniqueness of \eqref{eq:so2} since 
        \begin{align}
            \langle -(\mu \partial_x \varphi + \mu \varphi )+ \frac12 \langle \tilde{S}_2-\tilde{S}_1\rangle\rangle =0.
        \end{align}
        Moreover, we can also get the estimate 
        \begin{align}
            &|g^2(x)-g^2_{\infty}| \le N_{\beta}^2 e^{-\beta x}, \\
            &|\phi^2(x,\mu) - 4T^3_\infty g^2_\infty| \le C_1(\beta,\psi_b,T_b,M_{\alpha},N_{\beta}^2) e^{-\beta x},\quad \text{for any }\mu \in [-1,1],
        \end{align}
        with 
        \begin{align}
            N_{\beta}^2=&\frac{1}{\sqrt{2\beta}}\bigg(\frac{2}{3(1-\beta)}\int_0^1\mu G^2(0) d\mu \nonumber\\
            &+ \frac{2}{3(1-\beta)^2} \|e^{\beta x} ( -(\mu \partial_x \varphi + \mu \varphi )+ \frac12 \langle \tilde{S}_2-\tilde{S}_1\rangle)\|_{L^2(\mathbb{R}_+\times [-1,1])}^2 \bigg)^{\frac12}.
        \end{align}
        We can combine the above estimates and obtain  
        \begin{align}
            |g(x)-g_\infty| \le C e^{-\beta x},\quad |\phi(x,\mu)-4T_\infty^3 g_\infty|\le C e^{-\beta x},
        \end{align}
        with $C$ being a constant depending on $g_b,\phi_b,\beta,\psi_b,T_b,M_{\alpha}$.

    \end{proof}

\section{Monotonicity of linear transport equation}\label{Prooflemma2Psi}
\begin{lemma}\label{lm:psi}
    Given $\psi_b=\psi_b(\mu)$, for any $\mu \in (0,1]$ satisfying $0\le \psi_b \le \gamma$ for some constant $\gamma>0$. Let $h\in C([0,B])$ satisfy $0 \le h(x) \le \gamma$ for any $x\in [0,B]$. 
    Then there exists a unique solution $\psi \in C^1([0,B]\times [-1,1])$ to the equation 
    \begin{align}
        &\mu \partial_x \psi + \psi  = h, \text{ for }x\in (0,B),\quad \mu\in[-1,1],\label{eq:linearpsi}\\
        &\psi(0,\mu) = \psi_b(\mu),\quad \psi(B,\mu)=\psi(B,-\mu),\quad \text{ for any }\mu \in (0,1],
    \end{align}
    and the solution verifies $0\le \psi(x,\mu)\le \gamma$ for $x\in [0,B],\mu \in [-1,1]$.

    Furthermore, let $\psi_1,\psi_2$ be two solutions corresponding to the source data $ h_1, h_2$ and boundary data $\psi_{b1},\psi_{b2}$, respectively. If $0\le h_1(x) \le h_2(x)\le \gamma$ for any $x\in [0,B]$ and $0\le \psi_{b1}(\mu) \le \psi_{b2}(\mu) \le \gamma$ for $\mu \in [0,1]$, then $0\le \psi_1(x,\mu) \le \psi_2(x,\mu)\le \gamma$ for any $x\in [0,B]$ and $\mu \in [-1,1]$.
\end{lemma} 
\begin{proof}
    The existence and uniqueness of \eqref{eq:linearpsi} can be shown by the method of characteristic lines. Actually, we can compute the solution to \eqref{eq:linearpsi} directly. For $\mu>0$, we integrate from $[0,x]$ and get 
    \begin{align}\label{eq:sol/mu+}
        \psi(x,\mu) = \psi_b(\mu)e^{-\frac{x}{\mu}} + \int_0^x \frac{1}{\mu}h(s) e^{-\frac{x-s}{\mu}} ds.
    \end{align}
    For $\mu<0$ we integrate \eqref{eq:linearpsi} over $[x,B]$ and use the boundary condition $\psi(B,\mu)=\psi(B,-\mu)$ and get 
    \begin{align}\label{eq:sol/mu-}
       \psi(x,\mu) =  \psi_b(-\mu) e^{\frac{2B-x}{\mu}} - \int_0^B \frac{1}{\mu} e^{\frac{2B-x-s}{\mu}} h(s) ds - \int_x^B\frac{1}{\mu} e^{-\frac{x-s}{\mu}} h(s) ds,
    \end{align}
    The boundness of $\psi$ can be obtained directly by for $\mu>0$ from \eqref{eq:sol/mu+},
    \begin{align*}
        \psi(x,\mu) \le \gamma e^{-\frac{x}{\mu}} + \int_0^x \frac{\gamma}{\mu} e^{-\frac{x-s}{\mu}} ds = \gamma
    \end{align*}
    and for $\mu<0$ from \eqref{eq:sol/mu-},
    \begin{align*}
        \psi(x,\mu) \le \gamma e^{\frac{2B-x}{\mu}} - \int_0^B \frac{\gamma}{\mu} e^{\frac{2B-x-s}{\mu}} ds - \int_x^B \frac{\gamma}{\mu} e^{-\frac{x-s}{\mu}} ds = \gamma.
    \end{align*}
   
    The monotonic property follows directly from the expressions \eqref{eq:sol/mu+} and \eqref{eq:sol/mu-}. We can see from these expressions that $h_1(x) \le h_2(x)$ for all $x\in\mathbb{R}_+$ implies $\psi_1(x)\le \psi_2(x)$ for all $x\in\mathbb{R}_+$.
\end{proof}
\section{Formula for the transport equation}\label{lm:formula}
Suppose $h \in L^\infty(\mathbb{R}_{+})$, 
the solution to the linear transport equation 
\begin{align}
    &\mu \partial_x \psi + \psi = h, \quad \text{for } x\in [0,\infty],\, \mu \in [-,1,1],\\
    &\psi (0,\mu)=\psi_b,\quad \text{for }\mu\in (0,1]
\end{align}
is given by 
\begin{align}\label{eq:sol1/mu+}
    \psi(x,\mu) = \psi_b(\mu)e^{-\frac{x}{\mu}} + \int_0^x \frac{1}{\mu}h(s) e^{-\frac{x-s}{\mu}} ds.
\end{align}
and 
\begin{align}\label{eq:sol1/mu-}
    \psi(x,\mu) =    - \int_x^\infty \frac{1}{\mu} e^{-\frac{x-s}{\mu}} h(s) ds,
 \end{align}
 The formula \eqref{eq:sol1/mu+} is the same as \eqref{eq:sol/mu+} and can be shown the same way. The formula \eqref{eq:sol1/mu-} is obtained by sending $B\to\infty$ in \eqref{eq:sol/mu-}. The condition $h \in L^\infty$ is used, since for example,
 \begin{align}
    0\le \int_0^B \frac{1}{\mu} e^{\frac{2B-x-s}{\mu}} h(s) ds = e^{\frac{2B-x}{\mu}} \frac{1}{\mu} \int_0^B e^{-\frac{s}{\mu}} h(s) ds \le  e^{\frac{2B-x}{\mu}} \|h\|_{L^\infty([0,\infty))} (1-e^{-\frac{B}{\mu}})
 \end{align}
 By passing to the limit $B\to\infty$ the above inequality becomes 
 \begin{align}
    0\le  \lim_{B\to\infty} \int_0^B \frac{1}{\mu} e^{\frac{2B-x-s}{\mu}} h(s) ds \le 0
 \end{align}
 and hence $  \lim_{B\to\infty}\int_0^B\frac{1}{\mu}  e^{\frac{2B-x-s}{\mu}} h(s) ds =0$.

\section{Generalized Hardy's inequality}\label{Hardyinequality}

 First we recall the following  generalized Hardy's inequality.
\begin{lemma}[Generalized Hardy's inequality \cite{masmoudi2011hardy}] \label{lm:hardy}
	Let $p>1$ and $0<b\le \infty$. The inequality 
            \begin{align}\label{eq:hardy}
                \int_0^b \left(\int_0^x F(t) dt\right)^p u(x) dx \le C_{hd} \int_0^b F(x)^p v(x) dx
            \end{align} 
            holds for all measurable function $F(x) \ge 0$ on $(0,b)$ if and only if
            \begin{align*}
                A = \sup_{r\in (0,b)} \left(\int_r^b u(x) dx \right)^{\frac1p}\left(\int_0^r (v(x))^{1-p'} dx\right)^{\frac{1}{p'}} < \infty,
            \end{align*} 
            and the best constant of $C_{hd}$ satisfies $A \le C_{hd}^{1/p}\le p^{1/p}(p')^{1/p'}A$ where $\tfrac{1}{p'}+\tfrac{1}{p}=1$.
\end{lemma}

\section{Leray-Schauder Theorem}\label{LSTheorem}
\begin{theorem} \cite[Theorem 11.3]{gilbarg2015elliptic}
Let $\mathcal{N}$ be a compact mapping of a Banach space $D$ into itself, and
suppose there exists a constant $K$ such that
\[
\|x\|_{D}<K
\]
for all $x\in D$ and $\sigma \in [0,1]$ satisfying $x = \sigma \mathcal{N} x$. Then $\mathcal{N}$ has a fixed point.

\end{theorem}

    \bibliographystyle{siam}
    \bibliography{ReferenceHGM}
    
    \section*{Acknowledgements} 
The work of N.M. is supported by NSF grant DMS-1716466 and by Tamkeen under the NYU Abu Dhabi Research Institute grant of the center SITE. The work of M.G is supported by Tamkeen under the NYU Abu Dhabi Research Institute grant of the center SITE.

    \end{document}